\documentclass[final,times]{elsarticle}
\usepackage{commath}
\usepackage[english]{babel}
\usepackage{amsmath,amsthm}
\usepackage{amsfonts}
\usepackage{mathscinet}
\usepackage{textcomp}
\usepackage{units}
\usepackage{enumerate}
\usepackage{srcltx}
\usepackage{graphics}
\usepackage{color}
\usepackage{mathrsfs}
\usepackage{amsthm}
\usepackage{rotating}
\usepackage{tikz}
\usepackage{graphicx}
\usepackage[bookmarks=true, pdfstartview={FitH}]{hyperref}
\usepackage[margin=1in]{geometry}

\usepackage[arrow,matrix,curve]{xy}
\newdir{ >}{{}*!/-1em/\dir{>}}

\newcommand{\bZ}{\mathbb{Z}}

\newcommand{\lra}{\longrightarrow}

\newcommand{\os}{\overset}

\newcommand{\llm}{\underset{\longleftarrow}{\lim}\,}

\newcommand{\Hom}{\operatorname{Hom}}
\newcommand{\Ext}{\operatorname{Ext}}
\newcommand{\Ker}{\operatorname{Ker}}

\theoremstyle{dgthm}
\newtheorem{theorem}{Theorem}
\newtheorem{corollary}{Corollary}
\newtheorem{lemma}{Lemma}

\theoremstyle{dgdef}
\newtheorem{definition}{Definition}

\begin{document}


\title{On the universal coefficient formula and {\color{black} derived} $\varprojlim ^{(i)} $ functor }

\author{Anzor Beridze$^1$ and Leonard Mdzinarishvili$^2$ }

\address{$^1$Batumi Shota Rustaveli State University, 35 Ninoshvili str., Batumi, Georgia; e-mail: a.beridze@bsu.edu.ge}

\address{$^2$Georgian Technical University,
77, Kostava St., Tbilisi,
Georgia; e-mail:~l.mdzinarishvili@gtu.ge}


\begin{abstract} {\color{black} It is known that homology and inverse limit functors do not commute. In the paper we consider this very  problem and find its application for various homology theories. In particular, on the category of general topological spaces, there are defined exact homology functors induced by different non-free cochain complexes. Relation between them and other classical homology theories are given. In addition, for the defined homology functors the tautness and  the continuous properties are obtained.}
\end{abstract}


\maketitle
\section*{\bf Introduction}
{\color{black} From the beginning of 1960, there were many approaches to define exact homology theories using the methods of homological algebra (using an injective resolution) \cite{4}, \cite{8},\cite{Mas},\cite{Kuz}, \cite{14}.  These approaches gave the unique homology theory on the category of compact Hausdorff spaces \cite{Kuz}, \cite{14}. Our aim is to develop a method of homological algebra which gives opportunity to define on the category of general topological spaces a unique exact homology theory, generated by the given cochain complex. If  $H^*$ is the cohomology of the cochain complex $C^*=Hom(C_*;G)$, then the cohomology $H^*$ is said to be generated by the chain complex $C_*$. If a chain complex $C_*$ is free, then there is a universal coefficient formula of a cohomology theory \cite{5}, \cite{71}, \cite{15}. In the paper \cite{10}, using this formula and {\color{black} derived }inverse limit, a long exact sequence is written, which shows a relation of a cohomology of direct limit of chain complexes and inverse limit of cohomology groups of corresponding cochain complexes. The result for non-free chain complexes is extended in the paper \cite{12}.} {\color{black}  In the paper, we have considered the dual version of the paper by L. Mdzinarishvili and E. Spanier \cite{12}. In particular, Theorems 3 and 4 are equivalent to Theorems 3.4 and 4.1 in \cite{12}, while Theorems 1 and 2 correspond to assertions (1) and (2) respectively, which are presented without proofs in the same paper \cite{12}. In the paper, the proofs of Theorems 3 and 4  are presented in more detail, covering some minor gaps in the arguments of \cite{12}. Moreover, while the Universal Coefficients Formula, as formulated on page 291 in $\S 2$ of Chapter V of the paper \cite{Br}, includes the dual version of Theorem 1 in the case of constant cosheaves, we offer a new proof of it. Interested reader is refereed to Exercise 5.C.6 in \cite{15}, which is the same as the short exact sequence in the very last line of the paper \cite{Kel}. It can be proved by dualizing the proof of Theorem 5.3.3 of \cite{15}, which is the same as short exact sequence (12) of \cite{Kel}.} In addition, we obtain the result (Theorem 2), which shows that the considered approach of definition of homology groups induced by a cochain complex is a generalization of classical approach whenever the cochain complex is free. The main part of the paper is the application of the obtained results for various homology theories. In particular, on the category of general topological spaces, using the considered approach, three exact homology functors 
$\bar{H}_*^{M}(-;G)$, $\bar{H}_*(-;G)$ and $\bar{H}_*^{s}(-;G)$ are constructed that are induced by the Massey cohain,  the Alexander-Spanier covhain and singular cochain, respectively. Relation between them and  Steenrod, Milnor  or Borel-Moor homology theories are given. Moreover, for the defined homology functors, the tautness (Corollary 6, Corollary 7) and  the continuous (Corollary 8, Corollary 9, Corollary 10) properties are obtained. 
{\color{black} Note that Corollary 6 is related to a question implicitly raised in \cite{71}. In particular, on page 15 in $\S$1.3 of \cite{71}, two properties (6) and (7) are  formulated, indicating that the cohomology theory defined by Massey has a compact support. Later, on page 115, before the Theorem 4.22 is formulated, the following note is made:  "the following theorem may be looked on as sort of weak dual to properties (6)  and (7) of $\S$1.3". The reason to fix that as "sort of weak dual" is that the Theorem 4.22 does not cover the general case as it is in case of  the cohomology.  The Corollary 6 of the paper answers the question.}
 
~~\\

\section{\bf Universal Coefficient Formula}

{\color{black} Let $C^*$ be a cochain complex and  $\beta_{\#}:Hom(C^*;G') \to Hom(C^*;G'')$ be the chain map induced by $\beta: G' \to G''$, where $0 \to G \os{\alpha}{\lra} G' \os{\beta}{\lra} G'' \to 0 $ is an injective resolution of $G$. Consider the cone  $C_*(\beta_{\#})=\{ C_n(\beta_{\#}), \partial \}=\{Hom( C^*,\beta_{\#}), \partial \}$ of the chain map $\beta_{\#}$ (cf. \cite{12}), i.e.
	\begin{equation}\label{eq04}
		C_n(\beta_{\#}) \simeq Hom(C^{n};G') \oplus Hom(C^{n+1};G''),
	\end{equation}
	\begin{equation}\label{eq05}
		\partial(\varphi ',\varphi '')=(\varphi ' \circ \delta, \beta \circ \varphi ' -\varphi '' \circ \delta), ~~~ \forall (\varphi ',\varphi '') \in C_n(\beta_{\#}).
	\end{equation}
	Consequently, the homology group $\bar{H}_n=\bar{Z}_n/ \bar{B}_n$ is denoted by $\bar{H}_n=\bar{H}_n(C^*;G)=H_n(C_*(\beta_{\#}))$ and is called a homology with coefficient in $G$ generated by the cochain complex $C^*$. Note that if $f:C^* \lra C'^*$ is a homomorphism of cochain complexes, then it induces the homomorphism $\bar{f}:C'_* (\beta_{\#}) \lra C_*(\beta_{\#})$ of chain {\color{black}complexes}. In particular, for each $n \in \mathbb{Z}$ the homomorphism $\bar{f}_n: C'_n (\beta_{\#}) \lra C_n (\beta_{\#})$ is defined by the formula
	$\bar{f}_n(\varphi ', \varphi '')=(\varphi ' \circ f_n , \varphi \circ f_{n+1}).$  Consequently, it induces a homomorphism of homology groups $\bar{f}:\bar{H}_n(C'^*;G) \lra \bar{H}_n(C^*;G)$. Therefore, $\bar{H}_n$ is a naturally defined functor.}

In this section we will prove the Universal Coefficient Formula for a homology theory $\bar{H}_*$ generated by the given cochain complex $C^*$. 

\begin{theorem} [Universal Coefficient Formula]
For each cochain complex $C^*$ and $R$-module $G$ over a fixed principal ideal domain $R$, there exists a short exact sequence
\begin{equation}\label{eq3}
0 \lra \Ext(H^{n+1}(C^*);G) \os{\bar{\chi}}{\lra} {\bar H}_n(C^*;G) \os{\bar{\xi}}{\lra} \Hom(H^n(C^*);G) \lra 0.
\end{equation}
\end{theorem}
\begin{proof} We will define a homomorphism $\xi: \bar{Z}_n \lra Hom(H^n(C^*);G)$, which induces an epimorphism $\bar{\xi}:\bar{H}_n(C^*;G) \lra Hom(H^n(C^*);G)$. On the other hand, we will define a homomorphism  $\chi:Hom(H^{n+1}(C^*);G'') \lra \bar{H}_n(C^*;G)$ such that $\chi$ induces a  monomorphism $\bar{\chi}: Ext(H^{n+1}(C^*);G) \lra \bar{H}_n(C^*;G)$ and the short sequence \eqref{eq3} is exact.

{\bf a. There is a homomorphism $\xi:\bar{Z}_n \to Hom(H^n(C^*);G)$.} Let $(\varphi ' , \varphi '') \in \bar{Z}_n$ be a cycle, i.e. $\varphi' :C^n \to G'$ and $\varphi '': C^{n+1} \to G''$ are homomorphisms such that $\partial(\varphi ',\varphi '')=(\varphi ' \circ \delta, \beta \circ \varphi ' -\varphi '' \circ \delta)=(0,0)=0$ and therefore, the following diagram is commutative:
\begin{equation}\label{eq4}
\begin{tikzpicture}

\node (A) {$C^{n-1}$};
\node (B) [node distance=3cm, right of=A] {$C^n$};
\node (C) [node distance=3cm, right of=B] {$C^{n+1}$};

\node (B1) [node distance=2cm, below of=B] {$G'$};
\node (C1) [node distance=2cm, below of=C] {$G''$};

\draw[->] (A) to node [above]{$\delta$}(B);
\draw[->] (B) to node [above]{$\delta$}(C);

\draw[->] (B) to node [right]{$\varphi '$}(B1);
\draw[->] (C) to node [right]{$\varphi ''$}(C1);

\draw[->] (A) to node [above]{$0$}(B1);

\draw[->] (B1) to node [above]{$\beta$}(C1);

\end{tikzpicture}
\end{equation}
where $0$ is the zero map. Consider the groups of coboundaries $B^{n}$ and cocycles $Z^{n}$. Let $i: B^{n} \to Z^{n}$ and $j: Z^{n} \to C^{n}$ be natural monomorphisms and $\delta ' : C^{n-1} \to B^n$ be an epimophism induced by $\delta:C^{n-1} \to C^n$. Therefore, we have the following sequence
\begin{equation}\label{eq5}
C^{n-1} \os{\delta '}{\lra} B^n \os{i}{\lra} Z^n \os{j}{\lra} C^n \os{\delta}{\lra} C^{n+1},
\end{equation}
where $j \circ i \circ \delta ' = \delta$ and consequently $\delta \circ j =0$.  

Since $(\varphi',\varphi '') \in \bar{Z}_n$, we have the following commutative diagram

\begin{equation}\label{eq6}
\begin{tikzpicture}

\node (A) {$Z^n$};
\node (B) [node distance=3cm, right of=A] {$C^n$};
\node (C) [node distance=3cm, right of=B] {$C^{n+1}$};

\node (A1) [node distance=2cm, below of=A] {$G$};
\node (B1) [node distance=2cm, below of=B] {$G'$};
\node (C1) [node distance=2cm, below of=C] {$G''$};
\node (C0) [node distance=2cm, right of=C1] {$0~.$};
\node (A0) [node distance=2cm, left of=A1] {$0$};

\draw[->] (A) to node [above]{$j$}(B);
\draw[->] (B) to node [above]{$\delta$}(C);

\draw[->] (B) to node [right]{$\varphi '$}(B1);
\draw[->] (C) to node [right]{$\varphi ''$}(C1);

\draw[->] (A1) to node [above]{$\alpha$}(B1);
\draw[->] (B1) to node [above]{$\beta$}(C1);

\draw[->] (A0) to node [above]{}(A1);
\draw[->] (C1) to node [above]{}(C0);
\end{tikzpicture}
\end{equation}
Hence, $\varphi '' \circ \delta \circ j = \beta \circ \varphi ' \circ j$, and by the equality $\delta \circ j=0$, we obtain that $\beta \circ \varphi ' \circ j=0$.  So, $Im (\varphi ' \circ j) \subset Ker \beta = Im \alpha$. Therefore, there is a uniquely defined map $\varphi :Z^n \to G$ such that $\varphi ' \circ j =\alpha \circ \varphi$ (see the diagram \eqref{eq7}.

\begin{equation}\label{eq7}
\begin{tikzpicture}

\node (A) {$Z^n$};
\node (B) [node distance=3cm, right of=A] {$C^n$};
\node (C) [node distance=3cm, right of=B] {$C^{n+1}$};

\node (A1) [node distance=2cm, below of=A] {$G$};
\node (B1) [node distance=2cm, below of=B] {$G'$};
\node (C1) [node distance=2cm, below of=C] {$G''$};
\node (C0) [node distance=2cm, right of=C1] {$0~.$};
\node (A0) [node distance=2cm, left of=A1] {$0$};

\draw[->] (A) to node [above]{$j$}(B);
\draw[->] (B) to node [above]{$\delta$}(C);

\draw[dotted,->] (A) to node [right]{$\varphi $}(A1);
\draw[->] (B) to node [right]{$\varphi '$}(B1);
\draw[->] (C) to node [right]{$\varphi ''$}(C1);

\draw[->] (A1) to node [above]{$\alpha$}(B1);
\draw[->] (B1) to node [above]{$\beta$}(C1);

\draw[->] (A0) to node [above]{}(A1);
\draw[->] (C1) to node [above]{}(C0);
\end{tikzpicture}
\end{equation}
By the commutative diagram \eqref{eq4}, we have $\varphi ' \circ \delta=\varphi ' \circ j \circ i \circ \delta '=0$. Hence,  $ \alpha  \circ \varphi \circ i \circ \delta '   = \varphi '  \circ j \circ i \circ \delta' = \varphi ' \circ \delta=0$ (see the diagram \eqref{eq8}).  $\alpha$ is a monomorphism and so $ \varphi \circ i \circ \delta ' =0$. On the other hand, $\delta '$ is an epimorphism. Consequently, we have $  \varphi  \circ i =0$. Therefore, the homomorphism $\varphi : Z^n \to G$ induces a homomorphism  $\bar{\varphi} : H^n(C^*) \to G$ which belongs to $Hom(H^n(C^*);G)$. Hence, the following  diagram is commutative:
\begin{equation}\label{eq8}
\begin{tikzpicture}

\node (A) {$Z^n$};
\node (ABC) [node distance=1cm, right of=A] {};
\node (AA) [node distance=2cm, left of=A] {$B^n$};
\node (AB) [node distance=2cm, left of=AA] {$C^{n-1}$};
\node (B) [node distance=3cm, right of=A] {$C^n$};
\node (C) [node distance=3cm, right of=B] {$C^{n+1}$};

\node (A1) [node distance=2cm, below of=A] {$G$};
\node (ABC1) [node distance=1cm, below of=ABC] {$H^n$};
\node (B1) [node distance=2cm, below of=B] {$G'$};
\node (C1) [node distance=2cm, below of=C] {$G''$};
\node (C0) [node distance=2cm, right of=C1] {$0~,$};
\node (A0) [node distance=2cm, left of=A1] {$0$};

\draw[->] (AB) to node [above]{$\delta '$}(AA);
\draw[->] (AA) to node [above]{$i$}(A);
\draw[->] (A) to node [above]{$j$}(B);
\draw[->] (B) to node [above]{$\delta$}(C);

\draw[->] (A) to node [right]{$\varphi $}(A1);
\draw[->] (B) to node [right]{$\varphi '$}(B1);
\draw[->] (C) to node [right]{$\varphi ''$}(C1);
\draw[dotted,->] (AA) to node [right]{$0$}(A1);

\draw[->] (A1) to node [above]{$\alpha$}(B1);
\draw[->] (B1) to node [above]{$\beta$}(C1);

\draw[->] (A) to node [right]{$p$}(ABC1);
\draw[dotted,->] (ABC1) to node [right]{$\bar{\varphi}$}(A1);

\draw[->] (A0) to node [above]{}(A1);
\draw[->] (C1) to node [above]{}(C0);
\end{tikzpicture}
\end{equation}
where $H^n \equiv H^n(C^*).$ Let $\xi :\bar{Z}_n \to Hom(H^n;G)$ be the homomorphism defined by
\begin{equation}\label{eq9}
\xi(\varphi ', \varphi '')=\bar{\varphi}, ~~\forall (\varphi ', \varphi '') \in \bar{Z}_n.
\end{equation}

{\bf b. $\xi:\bar{Z}_n \to Hom(H^n(C^*);G)$ is an epimorphism.} Let $\bar{\varphi} \in Hom(H^n(C^*); G)$ be a homomorphism and $\varphi = \bar{\varphi} \circ p : Z^n \to G$ is the composition, where $p:Z^n \to H^n(C^*)$ is a projection. Let $\varphi ' :C^n \to G'$  be an extension of  $\alpha \circ \varphi  : Z^n \to G'$. In this case $\varphi ' \circ j = \alpha \circ \varphi$ and so $\beta \circ \varphi ' \circ j  =\beta \circ  \alpha \circ \varphi =0.$ Therefore, $\beta  \circ \varphi ' : C^n \to G''$ vanishes on the subgroup $Z^n$ and so it induces a homomorphism $\tilde{\varphi} '' :C^n/Z^n \simeq B^{n+1} \to G''$, which can be extended to a homomophism $\varphi '' : C^{n+1} \to G''$.  Since $\tilde{\varphi}'' \circ \delta '= \beta \circ \varphi '$, there is $\partial (\varphi ', \varphi '') = (\varphi ' \circ \delta ,\beta \circ \varphi ' - \varphi '' \circ \delta)=(\varphi ' \circ j \circ i \circ \delta' ,\beta \circ \varphi ' - \varphi '' \circ j \circ i \circ \delta')= (\alpha \circ \varphi \circ i \circ \delta' ,\beta \circ \varphi ' -  \tilde{\varphi}'' \circ \delta ' )=$ $(\alpha \circ \bar{\varphi} \circ p \circ i \circ \delta' , \beta \circ \varphi ' - \beta \circ \varphi ')=(0,0)=0$. Hence,  $(\varphi ', \varphi '') \in \bar{Z}_n$ and $\bar{\xi}(\varphi ', \varphi '')=\bar{\varphi}$ (see the diagram \eqref{eq10}).

\begin{equation}\label{eq10}
\begin{tikzpicture}

\node (A) {$B^n$};
\node (1A) [node distance=2cm, left of=A] {$C^{n-1}$};
\node (AA) [node distance=3cm, right of=A] {};
\node (B) [node distance=2cm, right of=A] {$Z^n$};
\node (C) [node distance=2cm, right of=B] {$C^n$};
\node (D) [node distance=2cm, right of=C] {$C^n/Z^n \simeq B^{n+1}$};
\node (E) [node distance=2cm, right of=D] {$C^{n+1}$};

\node (AA1) [node distance=1cm, below of=AA] {$H^n$};

\node (A1) [node distance=2cm, below of=A] {$0$};
\node (B1) [node distance=2cm, below of=B] {$G$};
\node (C1) [node distance=2cm, below of=C] {$G'$};
\node (D1) [node distance=2cm, below of=D] {$G''$};
\node (E1) [node distance=2cm, below of=E] {$0.$};

\draw[->] (AA1) to node [above]{$\bar{\varphi}$}(B1);
\draw[->] (B) to node [above]{$p$}(AA1);
\draw[->] (1A) to node [above]{$\delta '$}(A);
\draw[->] (A) to node [above]{$i$}(B);

\draw[->] (B) to node [above]{$j $}(C);
\draw[->] (C) to node [above]{$\delta' $}(D);
\draw[->] (D) to node [above]{$j \circ i$}(E);

\draw[->] (A1) to node [above]{$ $}(B1);
\draw[->] (B1) to node [above]{$\alpha$}(C1);
\draw[->] (C1) to node [above]{$\beta $}(D1);
\draw[->] (D1) to node [above]{$ $}(E1);

\draw[->] (B) to node [right]{$\varphi $}(B1);
\draw[dotted, ->] (C) to node [right]{$\varphi '$}(C1);
\draw[dotted, ->] (D) to node [right]{$\tilde{\varphi} ''$}(D1);
\draw[dotted, ->] (E) to node [right]{$\varphi ''$}(D1);

\end{tikzpicture}
\end{equation}

{\bf c. $\xi:\bar{Z}_n \to Hom(H^n(C^*);G)$ induces a homomorphism $\bar{\xi}:\bar{H}_n(C^*;G) \to Hom(H^n(C^*);G)$.} We have to show that the homomorphism $\xi :\bar{Z}_n \to Hom(H^n(C^*); G)$ vanishes on the subgroup $\bar{B}_n$. Indeed, let $(\psi ' , \psi '' ) \in C_{n+1}(\beta_\#)$ be an element. For $\partial (\psi ' , \psi '')=(\psi ' \circ \delta , \beta \circ \psi ' - \psi '' \circ \delta)=(\varphi ', \varphi '') \in \bar{B}_n \subset \bar{Z}_n$ we have $\varphi ' \circ j =0$. Indeed, $\varphi '  \circ j=\psi ' \circ  \delta \circ j=0$. Therefore, by the construction $\xi$, the homomorphism $\varphi : Z^n \to G$ corresponding to the pair $(\varphi ', \varphi '')$ satisfies the equation $\alpha \circ \varphi = \varphi ' \circ j =0 $ and so $\varphi =0$, because $\alpha$ is a monomorphism. Since $\varphi = \bar{\varphi} \circ p$ and $p$ is an epimorphism, we have $\bar{\varphi}=0$. Therefore, $\xi \partial (\psi ' , \psi '') = \xi (\varphi ' , \varphi '') = \bar{\varphi}=0 $  (see the diagram \eqref{eq11}).

\begin{equation}\label{eq11}
\begin{tikzpicture}
\node (A) {$Z^n$};
\node (AA) [node distance=1cm, right of=A] {};
\node (B) [node distance=3cm, right of=A] {$C^n$};
\node (C) [node distance=3cm, right of=B] {$C^{n+1}$};
\node (D) [node distance=3cm, right of=C] {$C^{n+2}$};

\node (A1) [node distance=2cm, below of=A] {$G$};
\node (AA1) [node distance=1cm, below of=AA] {$H^n$};
\node (B1) [node distance=2cm, below of=B] {$G'$};
\node (C1) [node distance=2cm, below of=C] {$G''$};
\node (C0) [node distance=2cm, right of=C1] {$0~.$};
\node (A0) [node distance=2cm, left of=A1] {$0$};

\draw[->] (A) to node [above]{$p$}(AA1);
\draw[->] (AA1) to node [above]{$\bar{\varphi}$}(A1);
\draw[->] (A) to node [above]{$j$}(B);
\draw[->] (B) to node [above]{$\delta$}(C);
\draw[->] (C) to node [above]{$\delta$}(D);

\draw[->] (C) to node [right]{$\psi ' $}(B1);
\draw[->] (D) to node [right]{$\psi '' $}(C1);

\draw[->] (A) to node [right]{$\varphi $}(A1);
\draw[->] (B) to node [right]{$\varphi '$}(B1);
\draw[->] (C) to node [right]{$\varphi ''$}(C1);

\draw[->] (A1) to node [above]{$\alpha$}(B1);
\draw[->] (B1) to node [above]{$\beta$}(C1);

\draw[->] (A0) to node [above]{}(A1);
\draw[->] (C1) to node [above]{}(C0);
\end{tikzpicture}
\end{equation}

{\bf d. The kernel of $\bar{\xi}:\bar{H}_n(C^*;G) \to Hom(H^n(C^*);G)$ is $Ext(H^{n+1}(C^*);G)$.} If we apply the functor $Hom(H^{n+1}(C^*);-)$  to the short exact sequence $0 \to G \os{\alpha}{\lra} G' \os{\beta}{\lra}  G'' \to 0$, then we obtain:
\begin{equation}\label{eq12}
 0 \lra Hom(H^{n+1}(C^*);G) \os{\alpha_*}{\lra} Hom(H^{n+1}(C^*);G') \os{\beta_*}{\lra} Hom(H^{n+1}(C^*);G'') \os{}{\lra} Ext(H^{n+1}(C^*);G) \lra 0.
\end{equation}
Therefore, we have the following isomorphism:
\begin{equation}\label{eq13}
Ext(H^{n+1}(C^*);G) \simeq Hom(H^{n+1}(C^*);G'')/Im \beta_*.
\end{equation}
Our aim is to define such a homomorpism $\chi : Hom(H^{n+1}(C^*);G'') \lra \bar{H}_n(C^*;G) $ that the following sequence is exact:
\begin{equation}\label{eq14}
Hom(H^{n+1}(C^*);G') \os{\beta_*}{\lra} Hom(H^{n+1}(C^*);G'') \os{\chi}{\lra} \bar{H}_n(C^*;G)  \os{\bar{\xi}}{\lra} Hom(H^n(C^*);G) \lra 0.
\end{equation}
Indeed, in this case, it is clear that for the homomorphisms $\bar{\xi}$, $\chi$ and $\beta_*$, we have the following short exact sequences:
\begin{equation}\label{eq15}
0 \lra Ker \bar{\xi} \os{}{\lra} \bar{H}_n(C^*;G) \os{\bar{\xi}}{\lra} Hom(H^n(C^*);G) \os{}{\lra}  0,
\end{equation}
\begin{equation}\label{eq16}
0 \lra Ker \chi \os{}{\lra} Hom(H^{n+1}(C^*);G'') \os{\chi}{\lra} Im \chi \os{}{\lra}  0,
\end{equation}
\begin{equation}\label{eq17}
0 \lra Ker \beta_* \os{}{\lra} Hom(H^{n+1}(C^*);G') \os{\beta_*}{\lra} Im \beta_* \os{}{\lra}  0.
\end{equation}
On the other hand, if we prove exactness of the sequence  \eqref{eq14}, then $Ker \bar{\xi} \simeq Im \chi$ and $Ker \chi \simeq Im \beta_*$. Therefore, we have:
\begin{equation}\label{eq18}
 Ker \bar{\xi} \simeq Im \chi \simeq  Hom (H^{n+1}(C^*);G'')/Ker \chi \simeq   Hom (H^{n+1}(C^*);G'')/Im \beta_* \simeq Ext(H^{n+1}(C^*);G).
\end{equation}

To define $\chi$, consider an element $\bar{\varphi}'' \in Hom(H^{n+1}(C^*);G'')$. Let $\varphi '' :C^{n+1} \lra G''$ be an extension of the composition $\bar{\varphi}'' \circ p : Z^{n+1} \lra G''$, where $p:Z^{n+1} \lra H^{n+1}(C^*)$ is a natural projection. In this case,  $\varphi'' \circ \delta =\varphi'' \circ    j \circ i \circ \delta ' = \tilde{\varphi}'' \circ p \circ i \circ \delta '  =0$ and so, if we take $\varphi ' =0$, then the following diagram is commutative:
\begin{equation}\label{eq19}
\begin{tikzpicture}

\node (A) {$C^{n-1}$};
\node (B) [node distance=3cm, right of=A] {$C^n$};
\node (C) [node distance=1.5cm, right of=B] {$B^{n+1}$};
\node (D) [node distance=1.5cm, right of=C] {$Z^{n+1}$};
\node (E) [node distance=1.5cm, right of=D] {$C^{n+1}$};

\node (B1) [node distance=2cm, below of=B] {$G'$};
\node (D1) [node distance=1cm, below of=D] {$H^{n+1}$};
\node (E1) [node distance=2cm, below of=E] {$G''$ .};

\draw[->] (D) to node [left]{$p$}(D1);
\draw[->] (D1) to node [right]{$\bar{\varphi}''$}(E1);

\draw[->] (A) to node [above]{$\delta$}(B);
\draw[->] (B) to node [above]{$\delta '$}(C);
\draw[->] (C) to node [above]{$i$}(D);
\draw[->] (D) to node [above]{$j$}(E);

\draw[->] (B) to node [right]{$\varphi '=0$}(B1);
\draw[->] (E) to node [right]{$\varphi ''$}(E1);

\draw[->] (A) to node [above]{$0$}(B1);

\draw[->] (B1) to node [above]{$\beta$}(E1);

\end{tikzpicture}
\end{equation}
Therefore, $(0,\varphi '') \in \bar{Z}_n$ and so, we can define $\chi$  in the following way:
\begin{equation}\label{eq20}
\chi(\bar{\varphi}'')= (0,- \varphi '')+\bar{B}_n,  ~~~\forall~ \bar{\varphi}'' \in Hom(H^{n+1}(C^*);G'').
\end{equation}
Let check that $\chi$ is well defined. Consider two different extensions $\varphi _1 ''$ and $\varphi _2 ''$ of the map $\bar{\varphi}'' \circ p :Z^{n+1} \lra G''$ and show that $(0, -\varphi_1 '')+\bar{B}_n=(0, -\varphi_2 '')+\bar{B}_n$. For this, we have to show that $(0, \varphi _2 '' - \varphi _1 '') \in \bar{B}_n$. Indeed, by the definition of $\varphi _1 ''$ and $\varphi _2 ''$, it is clear that $(\varphi _1 '' - \varphi _2 '') \circ j=\varphi _1 '' \circ j - \varphi _2 '' \circ j = \bar{\varphi}'' \circ p -\bar{\varphi}'' \circ p=0$ and so, $\varphi _1 '' - \varphi _2 ''$ induces a homomorphism $\psi : C^{n+1}/Z^{n+1} \lra G''$. On the other hand, $C^{n+1}/Z^{n+1} \simeq B^{n+2}$ and so, we have an extension $\psi '': C^{n+2} \lra G''$ of $\psi$ (see the diagran \eqref{eq261}). 
\begin{equation}\label{eq261}
\begin{tikzpicture}

\node (A) {$Z^{n+1}$};
\node (B) [node distance=3cm, right of=A] {$C^{n+1}$};
\node (C) [node distance=4cm, right of=B] {$C^{n+1}/Z^{n+1} \simeq B^{n+2}$};
\node (D) [node distance=4cm, right of=C] {$C^{n+2}$};

\node (B1) [node distance=2cm, below of=B] {$G''$};

\draw[->] (A) to node [above]{$j$}(B);
\draw[->] (B) to node [above]{$\delta '$}(C);
\draw[->] (C) to node [above]{$j \circ i$}(D);

\draw[->] (A) to node [right]{$0$}(B1);
\draw[->] (B) to node [right]{$\varphi_1 ' - \varphi_2 ''$}(B1);
\draw[->] (C) to node [above]{$\psi $}(B1);
\draw[->] (D) to node [above]{$\psi ''$}(B1);

\end{tikzpicture}
\end{equation}
In this case, it is easy to see that
\begin{equation}\label{eq21}
\partial(0, \psi '')=(0, -\psi '' \circ \delta ) =(0, - \psi '' \circ j \circ i \circ \delta ' )=(0, -\psi  \circ \delta ' )=(0, \varphi _2 '' -\varphi _1 ''). 
\end{equation}
Therefore, it remains to show that $Im \chi \simeq Ker \bar{\xi}$ and $Im \beta _* = Ker \chi$.
 
 {\bf d$_1$. $Im \chi \simeq Ker \bar{\xi}$.} Let $\bar{\varphi}'' \in Hom(H^{n+1}(C^*);G'')$ be an element, then $\bar{\xi} \left(\chi (\bar{\varphi}'')\right)= \bar{\xi} \left((0, -\varphi '')+\bar{B}_{n}\right)=\bar{\varphi}$. On the other hand, by construction of $\xi$ and the fact that the first coordinate of the pair $(0,-\varphi '')$ is zero, it is easy to check that $\bar{\varphi}=0$. Therefore, $Im \chi \subset Ker \bar{\xi}.$ Now consider an element $\bar{h} \in Ker \bar{\xi}$ and any of its representatives $(\varphi ', \varphi '') \in \bar{Z_n}$. In this case, by the definition of $\bar{\xi}$, there exists $\varphi :Z^n \to G$ such that the following diagram is commutative:
 
 \begin{equation}\label{eq22}
 \begin{tikzpicture}
 
 \node (A) {$Z^n$};
 \node (AA) [node distance=1cm, right of=A] {};
 \node (B) [node distance=3cm, right of=A] {$C^n$};
 \node (C) [node distance=3cm, right of=B] {$C^{n+1}$};
 
 \node (A1) [node distance=2cm, below of=A] {$G$};
 \node (AA1) [node distance=1cm, below of=AA] {$H^n$};
 \node (B1) [node distance=2cm, below of=B] {$G'$};
 \node (C1) [node distance=2cm, below of=C] {$G''$};
 \node (C0) [node distance=2cm, right of=C1] {$0~.$};
 \node (A0) [node distance=2cm, left of=A1] {$0$};
 
 \draw[->] (A) to node [above]{$p$}(AA1);
 \draw[dotted,->] (AA1) to node [above]{$\bar{\varphi}$}(A1);
 
 \draw[->] (A) to node [above]{$j$}(B);
 \draw[->] (B) to node [above]{$\delta$}(C);
 
 \draw[dotted,->] (A) to node [right]{$\varphi $}(A1);
 \draw[->] (B) to node [right]{$\varphi '$}(B1);
 \draw[->] (C) to node [right]{$\varphi ''$}(C1);
 
 \draw[->] (A1) to node [above]{$\alpha$}(B1);
 \draw[->] (B1) to node [above]{$\beta$}(C1);
 
 \draw[->] (A0) to node [above]{}(A1);
 \draw[->] (C1) to node [above]{}(C0);
 \end{tikzpicture}
 \end{equation}
Moreover, $\bar{h} \in Ker \bar{\xi}$ means that the homomorphism $\bar{\varphi} :H^n \to G$ induced by $\varphi   :Z^n \to G$ is zero. Therefore, $\varphi =0$ and so $\varphi ' \circ j = \alpha \circ \varphi =0$. Consequently, $\varphi ' :C^n \to G'$ induces a homomorphism $ \tilde{\psi} ' : C^n / Z^n \simeq B^{n+1} \to G'$. Let $\psi ' : C^{n+1} \to G'$ be an extension of $ \tilde{\psi}' \circ p :Z^{n+1} \to G'$ and $\psi '' = \beta \circ \psi '$ (see the diagram \eqref{eq23}).  In this case, the homomorphism $\psi = \psi '' - \varphi ''  : C^{n+1} \to G''$ vanishes on the $B^{n+1}$. Indeed,  $\psi \circ j \circ i \circ \delta ' = \psi '' \circ j \circ i \circ \delta ' -\varphi '' \circ j \circ i \circ \delta '  =\beta \circ \tilde{\psi}' \circ \delta ' -  \beta \circ \varphi ' = \beta \circ \varphi '-\beta \circ \varphi '=0$. On the other hand, $\delta '$ is an epimorphim and so $\psi \circ j \circ i =0$. Therefore,  $\psi \circ j : Z ^{n+1} \to G''$ induces a homomorphism $\bar{\psi}: H^{n+1} \to G''$ (see the diagram \eqref{eq23})

 \begin{equation}\label{eq23}
\begin{tikzpicture}

\node (A) {$Z^n$};
\node (B) [node distance=3cm, right of=A] {$C^n$};
\node (C) [node distance=2.5cm, right of=B] {$ B^{n+1}$};
\node (D) [node distance=2cm, right of=C] {$Z^{n+1}$};
\node (E) [node distance=1.5cm, right of=D] {$H^{n+1}$};
\node (F) [node distance=1cm, right of=E] {};
\node (F1) [node distance=1cm, below of=D] {$C^{n+1}$};

\node (A1) [node distance=2cm, below of=A] {$G$};
\node (B1) [node distance=2cm, below of=B] {$G'$};
\node (E1) [node distance=2cm, below of=E] {$G''$};
\node (E0) [node distance=2cm, right of=E1] {$0~.$};
\node (A0) [node distance=2cm, left of=A1] {$0$};

\draw[->] (D) to node [left]{$j$}(F1);
\draw[->] (F1) to node [above]{$\psi ''$}(E1);
\draw[<-] (9+0.3-0.75,-2.2-0.2+0.5) to node [left]{$\varphi ''$} (7.5+0.3,-1.2-0.2);

\draw[->] (A) to node [above]{$j$}(B);
\draw[->] (B) to node [above]{$\delta '$}(C);
\draw[->] (C) to node [above]{$i$}(D);
\draw[->] (D) to node [above]{$p$}(E);

\draw[dotted,->] (A) to node [right]{$\varphi $}(A1);
\draw[->] (B) to node [right]{$\varphi '$}(B1);
\draw[->] (E) to node [right]{$\bar{\psi}$}(E1);

\draw[->] (C) to node [above]{$\tilde {\psi} '$}(B1);
\draw[->] (F1) to node [above]{$ {\psi} '$}(B1);

\draw[->] (A1) to node [above]{$\alpha$}(B1);
\draw[->] (B1) to node [above]{$\beta$}(E1);

\draw[->] (A0) to node [above]{}(A1);
\draw[->] (E1) to node [above]{}(E0);
\end{tikzpicture}
\end{equation}
Our aim is to show that $\chi(\bar{\psi})=\bar{h}=(\varphi ', \varphi '')+\bar{B}_n$. Indeed, by the definition of $\chi$, it is easy to see that $\chi(\bar{\psi})=(0, - \psi)+\bar{B}_n$. Therefore, we have to show that $(\varphi ', \varphi '')-(0, -\psi)=(\varphi ', \varphi '' +\psi)=(\varphi ', \psi '') \in \bar{B}_n$. Indeed,
\begin{equation}\label{eq24}
\partial(\psi ', 0)=(\psi ' \circ \delta, \beta \circ \psi ')=(\psi ' \circ j \circ i \circ \delta ', \beta \circ \psi ')=(\tilde{\psi}' \circ \delta ', \psi '')=(\varphi ', \psi '').  
\end{equation}

{\bf d$_2$. $Im \beta_* \simeq Ker \chi$.} Let $\bar{\varphi}' \in Hom(H^{n+1}(C^*);G')$ be an element and $\varphi '' :C^{n+1} \to G''$ be an extension of the composition $\beta \circ \bar{\varphi}' \circ p :Z^{n+1} \to G''$. In this case we have $\partial(0,-\varphi '')=(0,\varphi '' \circ \delta)=(0,\varphi '' \circ j \circ i \circ \delta')= (0,\beta \circ  \bar{\varphi} ' \circ p \circ i \circ \delta')=(0,0)=0$ (see the diagram \eqref{eq26}). Therefore, $(0, -\varphi '') \in \bar{Z}_n$ and so we have
\begin{equation}\label{eq25}
\left(\chi \circ \beta_* \right) \left(\bar{\varphi}'\right) =\chi \left( \beta_* \left(\bar{\varphi}'\right) \right) =\chi \left( \beta \circ  \bar{\varphi }' \right) =
\left(0, -\varphi'' \right)+\bar{B}_n.  
\end{equation}
Our aim is to show that $\left(0, -\varphi'' \right) \in \bar{B}_n$. Indeed, let $\varphi ' :C^{n+1} \lra G'$ be an extension of the composition $\bar{\varphi}' \circ p : Z^{n+1} \lra G'$. In this case $( \beta \circ \varphi ' -\varphi '') \circ j= \beta \circ \varphi ' \circ j - \varphi '' \circ j =\beta \circ \bar{\varphi}' \circ p - \beta \circ \bar{\varphi}' \circ p =0$ and so $ \beta \circ \varphi ' -\varphi '' :C^{n+1} \lra G'' $  induces a homomorphism $\tilde{\psi}'': C^{n+1}/Z^{n+1}\simeq B^{n+2} \lra G''$  such that $\beta \circ \varphi ' - \varphi '' =\tilde{\psi}'' \circ \delta '$. Let $\psi '' :C^{n+2} \lra G''$ be an extension of a homomorphism $\tilde{\psi}':B^{n+2} \lra G''$ (see the diagram \eqref{eq26} ). 
 \begin{equation}\label{eq26}
\begin{tikzpicture}

\node (A) {$Z^{n+1}$};
\node (AA) [node distance=2cm, left of=A] {$B^{n+1}$};
\node (AAA) [node distance=2cm, left of=AA] {$C^n$};
\node (B) [node distance=2cm, right of=A] {$C^{n+1}$};
\node (CC) [node distance=4cm, right of=B] {$C^{n+1}/Z^{n+1} \simeq B^{n+2}$};
\node (C) [node distance=4cm, right of=CC] {$C^{n+2}$};

\node (AA1) [node distance=1cm, below of=A] {$H^{n+1}$};

\node (B1) [node distance=2cm, below of=B] {$G'$};
\node (C1) [node distance=2cm, below of=C] {$G''$};

\draw[->] (AAA) to node [above]{$\delta '$}(AA);
\draw[->] (AA) to node [above]{$i$}(A);
\draw[->] (A) to node [above]{$j$}(B);
\draw[->] (B) to node [above]{$\delta '$}(CC);
\draw[->] (CC) to node [above]{$j \circ i$}(C);

\draw[->] (A) to node [right]{$p$}(AA1);
\draw[->] (AA1) to node [above]{$\bar{\varphi}'$}(B1);
\draw[dotted, ->] (B) to node [right]{$\varphi '$}(B1);
\draw[dotted, ->] (C) to node [right]{$\psi ''$}(C1);

\draw[dotted, ->] (B) to node [above]{$ \varphi '' $}(C1);
\draw[dotted,->] (CC) to node [above]{$\tilde {\psi} ''$}(C1);

\draw[->] (B1) to node [above]{$\beta$}(C1);

\end{tikzpicture}
\end{equation}
In this case, we have
\begin{equation}\label{eq27}
\partial (-\varphi',-\psi'')=(-\varphi' \circ \delta,  - \beta \circ \varphi' + \psi '' \circ \delta )=(-\varphi' \circ j \circ i \circ \delta',   - \beta \circ \varphi' + \tilde{\psi} '' \circ \delta' )=(-\bar{\varphi}' \circ p \circ i \circ \delta',  - \beta \circ \varphi' +(\beta \circ \varphi' -\varphi '') )=(0, - \varphi ''). 
\end{equation}
Therefore, $(0,-\varphi'')\in \bar{B}_n$ and so, $\chi \circ \beta_*=0$. Hence, $Im \beta_* \subset \Ker \chi$. Now consider an element $\bar{\varphi}'' \in Ker \chi$. Let $\varphi '' :C^{n+1} \to G''$ be an extension of the composition $\bar{\varphi}'' \circ p :Z^{n+1} \to G''$ (see the diagram \eqref{eq28}). Then, by $\chi(\bar{\varphi}'')=(0, -\varphi '')+\bar{B}_n=0$, there exists $(\psi',\psi'') \in C_{n+1}(\beta_\#)$ such that 
\begin{equation}\label{eq27}
\partial (\psi',\psi'')=(\psi' \circ \delta,  \beta \circ \psi' - \psi '' \circ \delta )=(0, - \varphi ''). 
\end{equation}
Therefore, $\psi ' \circ \delta =\psi ' \circ j \circ i \circ \delta ' = 0$. Since $\delta ':C^n \to B^{n+1}$ is an epimorphism, we have $\psi ' \circ j \circ i=0$ and  so $\psi ' \circ j :Z^{n+1} \to G'$ induces a homomorphism $\bar{\psi}' : H^{n+1}(C^*) \to G'$. On the other hand, by $  \beta \circ \psi' - \psi '' \circ \delta =  -\varphi ''$, we have $ - \varphi '' \circ j  = \beta \circ \psi' \circ j - \psi '' \circ \delta \circ j =\beta \circ \psi' \circ j$ (see the diagram \eqref{eq28}). Therefore, $\beta_*(\bar{\psi}')=\bar{\varphi}''$ and so, $ Ker \chi \subset Im \beta_*$ 

 \begin{equation}\label{eq28}
\begin{tikzpicture}

\node (A) {$C^{n+1}$};
\node (B) [node distance=2cm, right of=A] {};
\node (C) [node distance=5cm, right of=B] {$C^{n+2}$};

\node (B1) [node distance=1cm, below of=B] {$Z^{n+1}$};
\node (B2) [node distance=2cm, below of=B] {$H^{n+1}$};

\node (A1) [node distance=3cm, below of=A] {$G'$};
\node (C1) [node distance=3cm, below of=C] {$G''$};

\draw[->] (A) to node [above]{$\delta $}(C);
\draw[->] (A1) to node [above]{$\beta $}(C1);

\draw[->] (B1) to node [above]{$j$}(A);
\draw[->] (B1) to node [right]{$p$}(B2);
\draw[->] (B2) to node [above]{$\bar{\psi}'$}(A1);
\draw[->] (A) to node [right]{$\psi'$}(A1);
\draw[->] (C) to node [right]{$\psi''$}(C1);
\draw[->] (B2) to node [above]{$\bar{\varphi} ''$}(C1);

\end{tikzpicture}
\end{equation}
\end{proof}

Since for each injective group $G$, a group of extensions $Ext(-;G)$ is trivial, by the exact sequence \eqref{eq3} we obtain the following corollary (cf. Lemma VII.4.4 \cite{71})

\begin{corollary}
	If $G$ is an injective, then there is an isomorphism
	\begin{equation}\label{eq2.0}
	\bar{H}_n(C^*;G) \simeq Hom(H^n(C^*);G).
	\end{equation}
\end{corollary}

Let $C_*=Hom(C^*;G)$ be a chain complex, where $C_n=Hom(C^n;G)$ and $\partial$ {\color{black} is defined} by $\partial(\varphi)= \varphi \circ \delta,$ for $\varphi \in Hom(C^n;G)$.  In this case, there is  a map $\alpha_*:Hom(C^*;G) \lra Hom(C^*;\beta_\#)$ defined by:
\begin{equation}\label{eq2.1}
\alpha_*(\varphi)=(\alpha \circ \varphi, 0), ~~~ \forall \varphi \in Hom(C^n;G).
\end{equation}
Let $H_n(C^*;G)$ be a homology group of chain complex $Hom(C^*;G)$.
\begin{theorem}
	If a cochain complex $C^*$ is free, then the homomorphism $\alpha_*:Hom(C^*;G) \lra Hom(C^*;\beta_\#)$ induces an isomoprhism
	\begin{equation}\label{eq2.2}
	\bar{\alpha}_*:H_n(C^*;G) \lra \bar{H}_n(C^*;G).
	\end{equation}
\end{theorem}  
\begin{proof}
	Since $C^*$ is a free cochain complex, there is a short exact sequence:
	 \begin{equation}\label{eq2.3}
	 0 \lra \Ext(H^{n+1}(C^*);G) \os{\tilde{\chi}}{\lra} { H}_n(Hom(C^*;G)) \os{\tilde{\xi}}{\lra} \Hom(H^n(C^*);G) \lra 0.
	 \end{equation}
{\color{black}	Let us review} how the morphisms $\tilde{\xi}$ and $\tilde{\chi}$ are defined according to W. Massey's \cite{71} approach. Note that Massey has considered a free chain complex case and consequently, he has obtained {\color{black}Universal} Coefficient Formula for cohomology theory and not homology theory.\\
	{\bf a.}  For each $\bar{\varphi} \in H_n(Hom(C^*;G)),$ element let $\tilde{\xi}(\bar{\varphi}): H^n(C^*) \lra G $ be homomorphism given by 
	\begin{equation}\label{eq2.4}
	\tilde{\xi}(\bar{\varphi})(\bar{c})= \langle \varphi, c \rangle =\varphi(c), ~~~~\forall~c \in \bar{c}, ~\bar{c} \in H^n(C^*),
	\end{equation}
	where $\varphi$ is a representative of $\bar{\varphi}$ \cite{8}.\\  
	{\bf b.} To define the homomorphism $\tilde{\chi}: Ext(H^{n+1}(C^*);G) \lra H_n(Hom(C^*;G))$, we need to use the isomorphism \eqref{eq13}. Consequently, the homomorphism $\tilde{\chi}$ is the homomorphism induced by  $\chi_0: Hom(H^{n+1}(C^*);G'') \lra H_n(Hom(C^*;G)),$ where $\chi_0$ is defined in the following way.  Let $\bar{\varphi}'' \in Hom(H^{n+1}(C^*);G'') $ be any element. Since $G''$ is an injective, there is  an extension $\varphi '' :C^{n+1} \lra G''$ of the composition $\bar{\varphi}'' \circ p :Z^{n+1} \lra G''.$  In this case $\partial(\varphi '')= \varphi '' \circ \delta= \varphi '' \circ j \circ i \circ \delta ' =\bar{\varphi}'' \circ p \circ i \circ \delta ' =0$  (see diagram \eqref{eq41.1}). 
	\begin{equation}\label{eq41.1}
	\begin{tikzpicture}
	
	\node (A) {$C^{n-1}$};
	\node (B) [node distance=3cm, right of=A] {$C^n$};
	\node (C) [node distance=2.5cm, right of=B] {$B^{n+1}$};
	\node (D) [node distance=2cm, right of=C] {$Z^{n+1}$};
	\node (E) [node distance=1.5cm, right of=D] {$C^{n+1}$};
	\node (F) [node distance=1cm, right of=E] {};
	\node (F1) [node distance=1cm, below of=D] {$H^{n+1}$};
	
	\node (A1) [node distance=2cm, below of=A] {$G$};
	\node (B1) [node distance=2cm, below of=B] {$G'$};
	\node (E1) [node distance=2cm, below of=E] {$G''$};
	\node (E0) [node distance=2cm, right of=E1] {$0~.$};
	\node (A0) [node distance=2cm, left of=A1] {$0$};
	
	\draw[->] (D) to node [left]{$p$}(F1);
	\draw[->] (F1) to node [above]{$\bar{\varphi}''$}(E1);
	
	\draw[->] (A) to node [above]{$j$}(B);
	\draw[->] (B) to node [above]{$\delta '$}(C);
	\draw[->] (C) to node [above]{$i$}(D);
	\draw[->] (D) to node [above]{$j$}(E);
	
	\draw[dotted,->] (B) to node [right]{$\varphi $}(A1);
	\draw[dotted, ->] (B) to node [right]{$\varphi '$}(B1);
	\draw[->] (E) to node [right]{$\varphi ''$}(E1);
	
	\draw[dotted, ->] (E) to node [above]{$ {\psi} '$}(B1);
	
	\draw[->] (A1) to node [above]{$\alpha$}(B1);
	\draw[->] (B1) to node [above]{$\beta$}(E1);
	
	\draw[->] (A0) to node [above]{}(A1);
	\draw[->] (E1) to node [above]{}(E0);
	\end{tikzpicture}
	\end{equation}	
	Therefore, it {\color{black} defines a} homology class $[\varphi''] \in H_{n+1}(Hom(C^*;G''))$. Let $E:H_{n+1}(Hom(C^*;G)) \lra H_{n}(Hom(C^*;G))$ be a boundary homomorphism induced by the following exact sequence:
	\begin{equation}\label{eq2.5}
	0 \lra Hom(C^*;G) \os{\alpha_\#}{\lra} Hom(C^*;G') \os{\beta_{\#}}{\lra} Hom(C^*;G'') \lra 0.
	\end{equation}
	Define a homomorphism $\chi_0 $ by the formula
	\begin{equation}\label{eq2.6}
    \chi_0\left(\bar{\varphi}''\right)=E\left([{\varphi}'']\right),~~~\forall~\bar{\varphi}'' \in Hom(H^{n+1}(C^*);G'').
	\end{equation}
	Note that the homomorphism $\chi_0 : Hom(H^{n+1}(C^*);G) \lra H_n(Hom(C^*;G))$ is a composition of  the isomorphism  $Hom(H^{n+1}(C^*);G) \os{\simeq}{\lra} H_{n+1}(Hom(C^*;G))$ and the homomorphism $E:H_{n+1}(Hom(C^*;G)) \lra H_{n}(Hom(C^*;G))$. To write the explicit formula for $\chi_0$, consider such a map $\psi':C^{n+1} \to G'$ that $\beta \circ \psi ' =\varphi ''$ (this is possible, because the cochain complex $C^*$ is free). Let $\varphi ' = \psi ' \circ \delta :C^n \to G'$, then we have $\beta \circ \psi ' = \varphi '' \circ \delta = 0$. Therefore, $\varphi ' \in Ker \beta = Im \alpha$ and so, there exists a unique map $\varphi :C^n \to G $ such that $\alpha \circ \varphi = \varphi '$. In this case, we have $ \alpha \circ \varphi  \circ \delta = \varphi ' \circ \delta = \psi ' \circ \delta \circ \delta =0$ and so, $\varphi \circ \delta =0$ because $\alpha$ is a monomorphism. On the other hand, $\partial (\varphi)= \varphi \circ \delta =0.$ Consequently, $\varphi$ defines a homology class $[\varphi] \in H_n((Hom(C^*;G))$ (see diagram \eqref{eq41.1}). Finally, by the formula \eqref{eq2.6} we have
	\begin{equation}\label{eq2.61}
		\chi_0\left(\bar{\varphi}''\right)=E\left([{\varphi}'']\right)=[\varphi],~~~\forall~\bar{\varphi}'' \in Hom(H^{n+1}(C^*);G'').
    \end{equation}

	In this case, the sequence \eqref{eq2.3} is induced by the following sequence
	\begin{equation}\label{eq2.7}
	Hom(H^{n+1}(C^*);G') \os{\beta_*}{\lra} Hom(H^{n+1}(C^*);G'') \os{\chi_0}{\lra} { H}_n(Hom(C^*;G)) \os{\tilde{\xi}}{\lra} Hom(H^n(C^*);G) \lra 0.
	\end{equation}
	Therefore, by \eqref{eq14} and \eqref{eq2.7}, it is sufficient to show that the following diagram is commutative:
	\begin{equation}\label{eq2.8}
	\begin{tikzpicture}
	
	\node (A) {$Hom(H^{n+1}(C^*);G')$};
	\node (B) [node distance=4cm, right of=A] {$Hom(H^{n+1}(C^*);G'')$};
	c\node (C) [node distance=4cm, right of=B] {$ H_n(Hom(C^*;G))$};
	\node (D) [node distance=4cm, right of=C] {$Hom(H^n(C^*);G)$};
	\node (E) [node distance=2cm, right of=D] {$0$};
	
	\node (A1) [node distance=2cm, below of=A] {$Hom(H^{n+1}(C^*);G')$};
	\node (B1) [node distance=2cm, below of=B] {$Hom(H^{n+1}(C^*);G'')$};
	\node (C1) [node distance=2cm, below of=C] {$ \bar{H}_n(Hom(C^*;\beta_{\#}))$};
	\node (D1) [node distance=2cm, below of=D] {$Hom(H^n(C^*);G)$};
	\node (E1) [node distance=2cm, below of=E] {$0$.};

	\draw[->] (A) to node [above]{$\beta_*$}(B);
	\draw[->] (B) to node [above]{$\chi_0$}(C);
	\draw[->] (C) to node [above]{$\tilde{\xi}$}(D);
	\draw[->] (D) to node [above]{}(E);
	
	\draw[->] (A1) to node [above]{$\beta_*$}(B1);
	\draw[->] (B1) to node [above]{$\chi$}(C1);
	\draw[->] (C1) to node [above]{$\bar{\xi}$}(D1);
	\draw[->] (D1) to node [above]{}(E1);
	
	\draw[->] (A) to node [right]{$1$}(A1);
	\draw[->] (B) to node [right]{$1$}(B1);
	\draw[->] (C) to node [right]{$\bar{\alpha}_*$}(C1);
	\draw[->] (D) to node [right]{$1$}(D1);
	\end{tikzpicture}
	\end{equation}
	Indeed, let $\bar{\varphi}'' \in Hom(H^{n+1}(C^*);G'')$ be an element and $\varphi '' :C^{n+1} \lra G''$ be an extension of the composition $\bar{\varphi}'' \circ p :Z^{n+1} \lra G''.$ Then, {\color{black} by the snake lemma}, we must take an element $\varphi' \in Hom(C^{n+1};G')$, such that $\beta_{\#}(\varphi')=\varphi ''$. Note that this is possible because of exactness of the sequence \eqref{eq2.5}. Then, there is a cycle $\varphi \in Hom(C^n;G)$, such that $\alpha_\#(\varphi)=\partial(\varphi').$ Let $[\varphi]=\varphi+B_n$ be the corresponding element in the homology group $H_n(Hom(C^*;G)),$ then $\chi_0(\bar{\varphi}'')=[\varphi].$ By the definition of the map $\tilde{\alpha}:H_n(Hom(C^*);G) \lra \bar{H}_n(Hom(C^*);\beta_{\#})$, we have
	\begin{equation}\label{eq2.10}
	\left(\tilde{\alpha}_* \circ \chi_0 \right) \left(\bar{\varphi}'' \right)=\tilde{\alpha}_* \left( \chi_0  \left(\bar{\varphi}'' \right) \right)=\tilde{\alpha}_* \left([\varphi] \right)=\left(\alpha \circ \varphi, 0 \right)+\bar{B}_n.
	\end{equation}
	On the other hand, by the definition of $\chi: Hom(G^{n+1}(C^*);G'') \lra \bar{H}_n(Hom(C^*);\beta_{\#}),$ we have
	\begin{equation}\label{eq2.11}
	 \chi(\bar{\varphi}'')=(0, - \varphi '')+\bar{B}_n,  ~~~\forall~ \bar{\varphi}'' \in Hom(H^{n+1};G'').
	\end{equation}
	Therefore, we have to show that $(\alpha \circ \varphi , 0)-(0, - \varphi'')=(\alpha \circ \varphi,  \varphi '') \in \bar{B_n}.$ Indeed, by the equality $\alpha_\#(\varphi)=\partial(\varphi ')$ and $\beta_{\#}(\varphi')=\varphi''$, we have $\alpha \circ \varphi=\varphi ' \circ \delta$ and $\beta \circ \varphi'=\varphi''$. Therefore,
	\begin{equation}\label{eq2.12}
	\partial (\varphi ', 0)=(\varphi ' \circ \delta , \beta \circ \varphi ')=(\alpha \circ \varphi, \varphi '').
	\end{equation}
	By \eqref{eq2.10}, \eqref{eq2.11}, and \eqref{eq2.12}, we obtain that $\tilde{\alpha}_* \circ \chi_0=\chi$. So, it remains to show that $\bar{\xi} \circ \tilde{\alpha}= \tilde{\xi}$.
	
	Let $[\varphi] \in H_n(Hom(C^*;G))$ be an element and $\varphi$ is its representative. Then, by the definitions of $\bar{\xi}$ and $\tilde{\alpha}_*$ we have
	\begin{equation}\label{eq2.13}
	\left(\bar{\xi} \circ \tilde{\alpha}_*\right)\left([\varphi] \right)=\bar{\xi} \left( \tilde{\alpha}_*\left([\varphi] \right) \right)=\bar{\xi} \left( \alpha \circ \varphi, 0 \right)=\left(\varphi \right)+\bar{B}_n.
	\end{equation}
	Therefore, if we take an element $\bar{c} \in H^n(C^*)$ and  any of its representatives $c \in \bar{c}$, then  by \eqref{eq2.13} we have
	
	\begin{equation}\label{eq2.14}
	\left(\bar{\xi} \circ \tilde{\alpha}\right)\left([\varphi] \right)\left(\bar{c}\right)=\left( \varphi +\bar{B}_n  \right)\left( \bar{c} \right)=\varphi(c).
	\end{equation}
	Therefore, by \eqref{eq2.4} , \eqref{eq2.13} and \eqref{eq2.14}, we obtain that
    \begin{equation}\label{eq2.15}
	\bar{\xi} \circ \tilde{\alpha}= \tilde{\xi}.
	\end{equation}
\end{proof}

Note that by the commutative diagram \eqref{eq2.8}, we obtain the following commutative diagram:

\begin{equation}\label{eq2.16}
\begin{tikzpicture}

\node (A) {$0$};
\node (B) [node distance=2.5cm, right of=A] {$Ext(H^{n+1}(C^*);G)$};
c\node (C) [node distance=4cm, right of=B] {$ H_n(Hom(C^*;G))$};
\node (D) [node distance=4cm, right of=C] {$Hom(H^n(C^*);G)$};
\node (E) [node distance=2cm, right of=D] {$0$};

\node (A1) [node distance=2cm, below of=A] {$0$};
\node (B1) [node distance=2cm, below of=B] {$Ext(H^{n+1}(C^*);G)$};
\node (C1) [node distance=2cm, below of=C] {$ \bar{H}_n(C^*;G)$};
\node (D1) [node distance=2cm, below of=D] {$Hom(H^n(C^*);G)$};
\node (E1) [node distance=2cm, below of=E] {$0$.};

\draw[->] (A) to node [above]{}(B);
\draw[->] (B) to node [above]{$\tilde{\chi}$}(C);
\draw[->] (C) to node [above]{$\tilde{\xi}$}(D);
\draw[->] (D) to node [above]{}(E);

\draw[->] (A1) to node [above]{}(B1);
\draw[->] (B1) to node [above]{$\bar{\chi}$}(C1);
\draw[->] (C1) to node [above]{$\bar{\xi}$}(D1);
\draw[->] (D1) to node [above]{}(E1);

\draw[->] (B) to node [right]{$1$}(B1);
\draw[->] (C) to node [right]{$\bar{\alpha}_*$}(C1);
\draw[->] (D) to node [right]{$1$}(D1);
\end{tikzpicture}
\end{equation}

Therefore, if a cochain complex $C^*$ is free, then the classical Universal Coefficient Formula is isomorphic to the Universal Coefficient Formula {\color{black}deduced} in this paper.

\section{Some {\color{black}properties}  {\color{black} of inverse limit and its derived}  functors}


As we have seen in the previous section, there exists an epimorphism $\xi:\bar{Z}_n \lra Hom(H^n(C^*);G)$ which induces a homomorphism:
\begin{equation}\label{eq32}
\bar{\xi}:\bar{H}_n(C^*;G) \lra Hom (H^n(C^*);G)
\end{equation}
and the following diagram is commutative:
\begin{equation}\label{eq33}
\begin{tikzpicture}

\node (A) {$\bar{Z}_n$};
\node (B) [node distance=2cm, right of=A] {};
\node (C) [node distance=2cm, right of=B] {$Hom(H^n(C^*);G)$};
\node (B') [node distance=2cm, below of=B] {$\bar{H}_n(C^*;G)$.};

\draw[->] (A) to node [above]{$\xi$}(C);
\draw[->] (A) to node [above]{$\bar{p}$}(B');
\draw[->] (B') to node [above]{$\bar{\xi}$}(C);
\end{tikzpicture}
\end{equation}
To investigate $Ker \xi$ we construct a homomorphism
\begin{equation}\label{eq34}
\omega:Hom(C^{n+1};G') \oplus Hom(C^{n+1}/B^{n+1};G'') \lra Ker \xi
\end{equation}
by $\omega(\psi', \psi '')=(\psi' \circ \delta, \beta \circ \psi' -\psi '' \circ q),$ where $q:C^{n+1} \lra C^{n+1}/B^{n+1}$ is the quotient  map. {\color{black} Let us show} that $\omega(\psi',\psi'') \in Ker \xi$. Indeed,  $\partial \omega(\psi',\psi'')= \partial (\psi' \circ \delta, \beta \circ \psi' -\psi '' \circ q)=(\psi \circ \delta \circ \delta, \beta \circ \psi' \circ \delta -(\beta \circ \psi ' - \psi'' \circ  q)\circ \delta)=(0,\beta \circ \psi' \circ \delta - \beta \circ \psi ' \circ \delta + \psi'' \circ  q \circ \delta )=(0,\psi'' \circ  q  \circ \delta )=(0,0),$ because $q \circ \delta=0.$ Hence, $\omega (\psi ', \psi '') \in \bar{Z}_n$. By the definition of $\xi : \bar{Z}_n \lra Hom(H^n(C^*);G)$, there exists a uniquely defined map $\varphi : Z^n \lra G$ such that  $\alpha \circ \varphi = \psi' \circ \delta \circ j$. On the other hand, $\delta \circ j =0$ and so, $\alpha \circ \varphi =0$, which induces a homomorphism $\bar{\varphi} :H^n(C^*) \lra G$. Note that $\alpha$ is a monomorphism and $\alpha \circ \varphi =0$ implies that $\varphi=0$ and consequently $\left(\omega(\psi ', \psi '')\right)=(\psi' \circ \delta, \beta \circ \psi' -\psi '' \circ q)=\bar{\varphi}=0$. Therefore, we obtain that $\omega(\psi ', \psi '') \in Ker \xi.$

\begin{lemma} For each integer $n \in \mathbb{N}$, there exists the following short exact sequence 
\begin{equation}\label{eq36}
	\begin{tikzpicture}
	
	\node (A) {$0$};
	\node (B) [node distance=2.5cm, right of=A] {$Hom(C^{n+1}/B^{n+1};G')$};
	\node (C) [node distance=5cm, right of=B] {$Hom(C^{n+1};G') \oplus Hom(C^{n+1}/B^{n+1};G'')$};
	\node (D) [node distance=4cm, right of=C] {$Ker \xi$};
	\node (E) [node distance=1.5cm, right of=D] {$0,$};

	\draw[->] (A) to node [above]{}(B);
	\draw[->] (B) to node [above]{$\sigma$}(C);
	\draw[->] (C) to node [above]{$\omega$}(D);
	\draw[->] (D) to node [above]{}(E);

	\end{tikzpicture}
\end{equation}
where $\sigma: Hom(C^{n+1}/B^{n+1};G') \lra Hom(C^{n+1};G') \oplus Hom(C^{n+1}/B^{n+1};G'')$ is defined by the formula
\begin{equation}\label{eq35}
\sigma(\varphi)=(\varphi \circ q, \beta \circ \varphi), ~~~\forall \varphi \in Hom(C^{n+1}/B^{n+1};G').
\end{equation} 
\end{lemma}

\begin{proof}
	{\bf a. $\omega$ is an epimorphism.} If $(\varphi ' , \varphi '') \in Ker \xi,$ then $\xi(\varphi ' , \varphi '')=\bar{\varphi}=0$ and so $\varphi:Z^n \lra G$ is zero as well. On the other hand, $\alpha \circ \varphi =\varphi ' \circ j=0$. Therefore, there is a unique homomorphism $\tilde{\varphi}':B^{n+1} \lra G'$, such that $\varphi ' = \tilde{\varphi}' \circ \delta '$. Let $\psi ' : C^{n+1} \lra G'$ be an extension of the map $\tilde{\varphi} ' : B^{n+1} \lra G'$ (see the diagram \eqref{eq37}).
\begin{equation}\label{eq37}
\begin{tikzpicture}

\node (A) {};
\node (B) [node distance=2cm, right of=A] {$Z^n$};
\node (C) [node distance=2cm, right of=B] {$C^n$};
\node (D) [node distance=2cm, right of=C] {$B^{n+1}$};
\node (E) [node distance=2cm, right of=D] {$C^{n+1}$};
\node (F) [node distance=2cm, right of=E] {};

\draw[->] (B) to node [above]{$j$}(C);
\draw[->] (C) to node [above]{$\delta '$}(D);
\draw[->] (D) to node [above]{$j \circ i$}(E);

\node (A1) [node distance=2cm, below of=A] {$0$};
\node (B1) [node distance=2cm, below of=B] {$G$};
\node (C1) [node distance=2cm, below of=C] {$G'$};
\node (E1) [node distance=2cm, below of=E] {$G''$};
\node (F1) [node distance=2cm, below of=F] {$0$.};

\draw[->] (A1) to node [above]{}(B1);
\draw[->] (B1) to node [above]{$\alpha$}(C1);
\draw[->] (C1) to node [above]{$\beta$}(E1);
\draw[->] (E1) to node [above]{}(F1);

\draw[->] (B) to node [left]{$\varphi$}(B1);
\draw[->] (C) to node [left]{$\varphi '$}(C1);
\draw[dotted, ->] (D) to node [above]{$\bar{\varphi}'$}(C1);
\draw[dotted,->] (E) to node [above]{$\psi '$}(C1);
\draw[->] (E) to node [right]{$\varphi ''$}(E1);
\end{tikzpicture}
\end{equation}
If we consider the map $\beta \circ \psi ' - \varphi '': C^{n+1} \lra G'',$ then by $(\varphi ', \varphi '') \in \bar{Z}_n,$ we have $(\beta \circ \psi ' - \varphi'') \circ \delta = \beta \circ \psi' \circ \delta - \varphi '' \circ \delta = \beta \circ \varphi ' - \varphi '' \circ \delta=0$.  Since $(\beta \circ \psi ' - \varphi'') \circ \delta = (\beta \circ \psi ' - \varphi '') \circ j \circ i \circ \delta ' =0$ and $\delta '$ is an epimorphism, there is $(\beta \circ \psi ' - \psi'') \circ j \circ i =0.$ Therefore, there is a homomorphism $\psi '' :C^{n+1}/B^{n+1} \lra G''$ such that $\beta \circ \psi ' -\varphi '' =\psi '' \circ q$ and so $\varphi '' = \beta \circ \psi ' - \psi '' \circ q$ (see the diagram \eqref{eq38}).

\begin{equation}\label{eq38}
\begin{tikzpicture}

\node (A) {};
\node (B) [node distance=2cm, right of=A] {$B^{n+1}$};
\node (C) [node distance=2cm, right of=B] {$C^{n+1}$};
\node (D) [node distance=2cm, right of=C] {$C^{n+1}/B^{n+1}$};
\node (E) [node distance=2cm, right of=D] {$0$};

\draw[->] (B) to node [above]{$j \circ i$}(C);
\draw[->] (C) to node [above]{$q$}(D);
\draw[->] (D) to node [above]{}(E);

\node (C1) [node distance=2cm, below of=C] {$G''$};

\draw[->] (C) to node [left]{$\beta \circ \psi '-\varphi ''$}(C1);
\draw[dotted, ->] (D) to node [above]{$\psi''$}(C1);
\end{tikzpicture}
\end{equation}
Hence, $(\psi ' , \psi'') \in Hom(C^{n+1};G') \oplus Hom(C^{n+1}/B^{n+1};G'')$ and $\omega(\psi ', \psi'')=(\psi ' \circ \delta, \beta \circ \psi ' - \psi '' \circ q)=(\varphi ', \varphi '')$. So,  $\omega$ is an epimorphism.

{\bf b. There is an equality  $Im \sigma = Ker \omega$.} By the definition, we have $(\omega \circ \sigma) (\varphi)=\omega ( \sigma (\varphi))= \omega (\varphi \circ q, \beta \circ \varphi)=(\varphi \circ q \circ \delta, \beta \circ \varphi \circ q - \beta \circ \varphi \circ q) =(0,0)=0,$ because $q \circ \delta =0$. Therefore, $Im \sigma \subset Ker \omega.$ On the other hand, if $(\psi', \psi'')  \in Ker \omega,$ then  $\omega(\psi ', \psi '')=(\psi ' \circ \delta, \beta \circ \psi' - \psi '' \circ q)=0$ and so, $\psi ' \circ \delta =0$ and $\beta \circ \psi' =\psi '' \circ q.$ On the other hand, $\psi ' \circ \delta = \psi ' \circ j \circ i \circ \delta '=0$. Therefore,  we have $\psi' \circ j \circ i =0,$ because $\delta'$ is an epimorphism. So, there is a unique homomorphism $\varphi:C^{n+1}/B^{n+1} \lra G'$ such that $\psi'=\varphi \circ q.$ In this case, $\beta \circ \varphi \circ q=\beta \circ \psi' = \psi'' \circ q$ and since $q$ is an epimorphism, $\beta \circ \varphi=\psi''.$ Therefore, $\sigma(\varphi)=(\psi', \psi'')$ and so, $Ker \omega \subset  Im \sigma.$

{\bf c. $\sigma$ is a monomorphism.} If $\sigma(\varphi)=(\varphi \circ q, \beta \circ \varphi)=0,$ i.e.  $\varphi \circ q=0$ and since  $q$ is an epimorphism, we have $\varphi=0$.

\end{proof}

Let $\mathbf{C}^*=\{C^*_\gamma \}$ be a direct system of cochain complexes. Consider the corresponding inverse system  $\mathbf{C}_*=\{C_*^\gamma(\beta^{\#})\}=\{Hom(C^*_\gamma;\beta^\#)\}$ of chain complexes.

\begin{lemma}
	For each direct system $\mathbf{C}^*=\{C^*_\gamma \}$  of cochain complexes, there is an isomorphism
	\begin{equation}\label{eq39}
	Hom(\varinjlim C^*_\gamma  ; \beta^\#) \simeq \varprojlim Hom(C^*_\gamma;\beta^\#).
	\end{equation}
\end{lemma} 

\begin{proof}
	Consider a chain complex
	\begin{equation}\label{eq40}
	Hom(\varinjlim C^*_\gamma;\beta^\#)=\{Hom(\varinjlim C^n_\gamma; G') \oplus Hom(\varinjlim C^{n+1}_\gamma;G''), \partial \},
	\end{equation} 
	where $\partial (\varphi ', \varphi '')= (\varphi ' \circ \delta , \beta \circ \varphi ' - \varphi '' \circ \delta ).$ Note that $\delta = \varinjlim \delta_\gamma : \varinjlim C_\gamma^{n-1} \lra \varinjlim C_\gamma^n,$ where $\delta_\gamma :C_\gamma^{n-1} \lra  C_\gamma^n$ is the coboundary map of the cochain complex $C^*_\gamma.$ Since for any $G$ there is an isomorphism  $Hom(\varinjlim C^*_\gamma;G) \simeq \varprojlim Hom(C^*_\gamma;G)$, we have
	\begin{equation}\label{eq41}	
	Hom(\varinjlim C_\gamma^n;G') \oplus Hom(\varinjlim C_\gamma^{n+1};G'') \simeq \varprojlim Hom(C_\gamma^n;G') \oplus \varprojlim Hom(C_\gamma^{n+1};G'').
	\end{equation}
\end{proof}

\begin{lemma}
	If $f^\#:C^* \lra C'^*$ is a homomorphism of cochain complexes, then there is a commutative diagram:
	\begin{equation}\label{eq42}	
	\begin{tikzpicture}
	
	\node (A) {$0$};
	\node (B) [node distance=2.5cm, right of=A] {$Hom(C'^{n+1}/B'^{n+1};G')$};
	\node (C) [node distance=5.5cm, right of=B] {$Hom(C'^{n+1};G') \oplus Hom(C'^{n+1}/B'^{n+1};G'')$};
	\node (D) [node distance=4.5cm, right of=C] {$Ker \xi'$};
	\node (E) [node distance=1.5cm, right of=D] {$0$};
	
	\draw[->] (A) to node [left]{}(B);
	\draw[->] (B) to node [above]{$\tau '$}(C);
	\draw[->] (C) to node [above]{$\mu '$}(D);
	\draw[->] (D) to node [left]{}(E);
	
	\node (A1) [node distance=2cm, below of=A] {$0$};
	\node (B1) [node distance=2cm, below of=B] {$Hom(C^{n+1}/B^{n+1};G')$};
	\node (C1) [node distance=2cm, below of=C] {$Hom(C^{n+1};G') \oplus Hom(C^{n+1}/B^{n+1};G'')$};
	\node (D1) [node distance=2cm, below of=D] {$Ker \xi$};
	\node (E1) [node distance=2cm, below of=E] {$0.$};
	
	\draw[->] (A1) to node [left]{}(B1);
	\draw[->] (B1) to node [above]{$\tau$}(C1);
	\draw[->] (C1) to node [above]{$\mu$}(D1);
	\draw[->] (D1) to node [left]{}(E1);
	
	\draw[->] (B) to node [left]{$\tilde{f}_\#$}(B1);
	\draw[->] (C) to node [left]{$(f_\#, \tilde{f}_\#)$}(C1);
	\draw[->] (D) to node [left]{$\tilde{f}$}(D1);

	\end{tikzpicture}
	\end{equation}
\end{lemma}
\begin{proof} Note that homomorphisms $\tilde{f}_{\#}:Hom(C'^{n+1}/B'^{n+1};G' ) \lra Hom(C^{n+1}/B^{n+1};G')$	and $(f_{\#},\tilde{f}_{\#}):Hom(C'^{n+1};G') \oplus Hom(C'^{n+1}/B'^{n+1};G'') \lra Hom(C^{n+1};G') \oplus Hom(C^{n+1}/B^{n+1};G'')$ are naturally defined by $\tilde{f}_\#(\varphi') = \varphi' \circ \tilde{f}_{n+1}$ and $(f_\#, \tilde{f}_\#)(\varphi ',\varphi '')=(\varphi ' \circ f_{n+1}, \varphi '' \circ \tilde{f}_{n+1}),$ where $\tilde{f}_{n+1}:C^{n+1}/B^{n+1} \lra C'^{n+1}/B'^{n+1}$ is induced by $f_{n+1}:C^{n+1} \lra C'^{n+1}.$
	
{\bf a. $\tau \circ \tilde{f}_\# = (f_\#, \tilde{f}_\#) \circ \tau '$.} By the definition, we have $\left(\tau \circ \tilde{f}_\# \right)(\varphi ')=\tau \left( \tilde{f}_\# (\varphi ') \right)=\tau \left(\varphi ' \circ \tilde{f}_{n+1}\right)=\left(\varphi ' \circ \tilde{f}_{n+1} \circ q, \beta \circ \varphi ' \circ \tilde{f}_{n+1} \right) $ and $\left((f_\#,\tilde{f}_\#) \circ \tau ' \right) \left(\varphi '\right) = \left(f_\#,\tilde{f}_\#\right) \left( \tau ' \left(\varphi '\right) \right)=\left(f_\#,\tilde{f}_\#\right) \left(\varphi ' \circ q',\beta \circ \varphi'\right)=(\varphi ' \circ  q' \circ f_{n+1},\beta \circ  \varphi ' \circ \tilde{f}_{n+1}).$  Since $\tilde{f}_{n+1} \circ q =q' \circ f_{n+1},$ we have $\varphi ' \circ \tilde{f}_{n+1} \circ q = \varphi ' \circ q' \circ f_{n+1}.$ Hence, $\tau \circ \tilde{f}_\# = (f_\#, \tilde{f}_\#) \circ \tau '$.  

{\bf b. $\mu \circ (f_\#, \tilde{f}_\#)=\tilde{f} \circ \mu '$.} By the definition, we have $\left(\mu \circ (f_\#, \tilde{f}_\#)\right)(\varphi ' , \varphi '' )=\mu \left((f_\#, \tilde{f}_\#)(\varphi ', \varphi '')\right)= \mu \left(\varphi ' \circ f_{n+1}, \varphi '' \circ \tilde{f}_{n+1}\right)=\left(\varphi ' \circ f_{n+1} \circ \delta, \beta \circ \varphi ' \circ f_{n+1}- \varphi'' \circ \tilde{f}_{n+1} \circ q \right)$ and $\left( \tilde{f} \circ \mu ' \right)(\varphi ', \varphi '')= \tilde{f} \left( \mu ' (\varphi ', \varphi '')\right)=\tilde{f} \left(\varphi ' \circ \delta ', \beta \circ \varphi' - \varphi '' \circ q' \right)=(\varphi ' \circ \delta ' \circ f_n, \beta \circ \varphi ' \circ f_{n+1}-\varphi '' \circ  q' \circ f_{n+1}).$ Since $\tilde{f}_{n+1} \circ q= q' \circ f_{n+1}$ and $\delta ' \circ f_n = f_{n+1} \circ \delta,$ there are {\color{black}equalities } $ \varphi ' \circ f_{n+1} \circ \delta = \varphi ' \circ \delta ' \circ f_n$ and  $\varphi '' \circ \tilde{f}_{n+1} \circ q= \varphi '' \circ q' \circ f_{n+1}.$ Hence, $\mu \circ (f_\#, \tilde{f}_\#)=\tilde{f} \circ \mu '$.
\end{proof}
Let $\{Hom(C_\gamma^{n+1};G') \oplus Hom(C_\gamma^{n+1}/B_\gamma^{n+1}; G'') \}$ be an inverse system generated by the direct system $\{C^{n+1}_\gamma \}$. It is clear that for each $\gamma$ there is an exact sequence
\begin{equation}\label{eq43}
\begin{tikzpicture}

\node (A) {$0$};
\node (B) [node distance=2.5cm, right of=A] {$Hom(C^{n+1}_\gamma;G')$};
\node (C) [node distance=5cm, right of=B] {$Hom(C^{n+1}_\gamma;G') \oplus Hom(C^{n+1}_\gamma/B^{n+1}_\gamma;G'')$};
\node (D) [node distance=5cm, right of=C] {$Hom(C^{n+1}_\gamma/B^{n+1}_\gamma;G'')$};
\node (E) [node distance=2.5cm, right of=D] {$0$.};

\draw[->] (A) to node [above]{}(B);
\draw[->] (B) to node [above]{$\tau$}(C);
\draw[->] (C) to node [above]{$\mu$}(D);
\draw[->] (D) to node [above]{}(E);

\end{tikzpicture}
\end{equation}
Hence, by the main property of the derived functors $\varprojlim ^{(i)}$ there is a long exact sequence:
\begin{equation}\label{eq44}
\begin{tikzpicture}

\node (A) {$\dots$};
\node (B) [node distance=2.5cm, right of=A] {$\varprojlim ^{(i)} Hom(C^{n+1}_\gamma;G')$};
\node (C) [node distance=5.5cm, right of=B] {$\varprojlim ^{(i)} \left( Hom(C^{n+1}_\gamma;G') \oplus Hom(C^{n+1}_\gamma/B^{n+1}_\gamma;G'')\right)$};
\node (D) [node distance=6cm, right of=C] {$\varprojlim ^{(i)} Hom(C^{n+1}_\gamma/B^{n+1}_\gamma;G'')$};
\node (E) [node distance=3cm, right of=D] {$\dots$.};

\draw[->] (A) to node [above]{}(B);
\draw[->] (B) to node [above]{$\tilde{\tau}$}(C);
\draw[->] (C) to node [above]{$\tilde{\mu}$}(D);
\draw[->] (D) to node [above]{}(E);

\end{tikzpicture}
\end{equation}
On the other hand, since for each injective {\color{black}group}  $G$,  $\varprojlim ^{(i)} \{Hom(C^{n+1};G)\}=0,$ $i\ge 1$ (see Lemma 1.3 \cite{6}), we obtain the following result. 
\begin{corollary}
	For each {\color{black}pair of injective groups}$G'$ and $G''$, there is the following equality
	\begin{equation}\label{eq45}
	{\varprojlim} ^ {(i)} \left( Hom(C^{n+1}_\gamma;G') \oplus Hom(C^{n+1}_\gamma/B^{n+1}_\gamma;G'') \right)=0, ~~ i\ge 1.
	\end{equation}
\end{corollary}

Using the obtained result, we will prove the following lemma.

\begin{lemma} 
	For each integer $i \ge1$, there is an equality
	\begin{equation}\label{eq46}
	{\varprojlim} ^ {(i)} Ker \xi_\gamma = 0.
	\end{equation}
\end{lemma}
\begin{proof}
	By Lemma 1, for each $\xi_\gamma$, there is  a short exact sequence
	\begin{equation}\label{eq47}
	\begin{tikzpicture}
	
	\node (A) {$0$};
	\node (B) [node distance=2.5cm, right of=A] {$Hom(C^{n+1}_\gamma/B^{n+1}_\gamma;G')$};
	\node (C) [node distance=5cm, right of=B] {$Hom(C^{n+1}_\gamma;G') \oplus Hom(C^{n+1}_\gamma/B^{n+1}_\gamma;G'')$};
	\node (D) [node distance=4cm, right of=C] {$Ker \xi_\gamma$};
	\node (E) [node distance=1.5cm, right of=D] {$0$.};

	\draw[->] (A) to node [above]{}(B);
	\draw[->] (B) to node [above]{$\sigma_\gamma$}(C);
	\draw[->] (C) to node [above]{$\omega_\gamma$}(D);
	\draw[->] (D) to node [above]{}(E);
	
	\end{tikzpicture}
	\end{equation}
By the main property of a derived functor $\varprojlim ^{(i)}$, there is a long exact sequence	
	
	\begin{equation}\label{eq48}
	\begin{tikzpicture}
	
	\node (A) {$\dots$};
	\node (B) [node distance=3cm, right of=A] {$ \varprojlim ^{(i)} Hom(C^{n+1}_\gamma/B^{n+1}_\gamma;G')$};
	\node (C) [node distance=6cm, right of=B] {$\varprojlim ^{(i)} \left( Hom(C^{n+1}_\gamma;G') \oplus Hom(C^{n+1}_\gamma/B^{n+1}_\gamma;G'') \right)$};
	\node (D) [node distance=5cm, right of=C] {$ \varprojlim ^{(i)}  \Ker \xi_\gamma $};
	\node (E) [node distance=2cm, right of=D] {$\dots$.};
	
	\draw[->] (A) to node [above]{}(B);
	\draw[->] (B) to node [above]{}(C);
	\draw[->] (C) to node [above]{}(D);
	\draw[->] (D) to node [above]{}(E);
	
	\end{tikzpicture}
	\end{equation}
	By Lemma 1.3 \cite{6} for each $i \ge 0,$ there is an equality ${\varprojlim} ^ {(i)} Hom(C_\gamma^{n+1}/B_\gamma^{n+1};G') = 0$ and by Corollary 2,  for each $i \ge 1$ we obtain
	\begin{equation}\label{eq49}
	{\varprojlim} ^ {(i)} \left(Hom (C_\gamma^{n+1};G') \oplus Hom(C_\gamma^{n+1}/B_\gamma^{n+1};G'') \right)= 0.
	\end{equation}
Hence, by the long exact sequence \eqref{eq48} we obtain that ${\varprojlim} ^ {(i)}  Ker \xi_\gamma =0$, $i\ge 1$.
\end{proof}
\begin{corollary}
	For each integer $i\ge1,$ there is an isomorphism
	\begin{equation}\label{eq50}
	{\varprojlim} ^ {(i)} \bar{Z}_n^\gamma \simeq  {\varprojlim} ^ {(i)} Hom(H^n(C^*_\gamma) ;G).
	\end{equation}
\end{corollary}
\begin{proof}
	By {\bf a.} of Theorem 1, there is an epimorphism $\xi_\gamma :\bar{Z}_n^\gamma \lra Hom(H^n(C^*_\gamma) ;G).$
	Therefore, the following sequence is exact
	\begin{equation}\label{eq51}
	\begin{tikzpicture}
	
	\node (A) {$0$};
	\node (B) [node distance=2cm, right of=A] {$Ker \xi_\gamma$};
	\node (C) [node distance=2cm, right of=B] {$\bar{Z}_n^\gamma$};
	\node (D) [node distance=2cm, right of=C] {$Hom(H^n(C^*_\gamma);G)$};
	\node (E) [node distance=2cm, right of=D] {$0$.};
	
	\draw[->] (A) to node [above]{}(B);
	\draw[->] (B) to node [above]{}(C);
	\draw[->] (C) to node [above]{$\xi_\gamma$}(D);
	\draw[->] (D) to node [above]{}(E);
	
	\end{tikzpicture}
	\end{equation}
	Consequently, it induces the following long exact sequence
	\begin{equation}\label{eq52}
	\begin{tikzpicture}
	
	\node (A) {$\dots$};
	\node (B) [node distance=2cm, right of=A] {$ {\varprojlim} ^ {(i)}  Ker \xi_\gamma $};
	\node (C) [node distance=2.5cm, right of=B] {${\varprojlim} ^ {(i)} \bar{Z}_n^\gamma $};
	\node (D) [node distance=3cm, right of=C] {$ {\varprojlim} ^ {(i)} Hom(H^n(C^*_\gamma);G) $};
	\node (E) [node distance=3.5cm, right of=D] {${\varprojlim} ^ {(i+1)}  Ker \xi_\gamma $};
	\node (F) [node distance=2cm, right of=E] {$\dots$.};
	
	\draw[->] (A) to node [above]{}(B);
	\draw[->] (B) to node [above]{}(C);
	\draw[->] (C) to node [above]{}(D);
	\draw[->] (D) to node [above]{}(E);
	\draw[->] (E) to node [above]{}(F);
	
	\end{tikzpicture}
	\end{equation}
	On the other hand, by Lemma 4,  ${\varprojlim} ^ {(i)} \{Ker \xi_\gamma \}= 0, ~~ i\ge 1.$ Therefore, for $i\ge 1,$ we have  an isomorphism ${\varprojlim} ^ {(i)} \bar{Z}_n^\gamma \simeq  {\varprojlim} ^ {(i)} Hom(H^n(C^*_\gamma) ;G).$
\end{proof}

Note that for each $\gamma$, there is a natural commutative triangle
	\begin{equation}\label{eq53}
	\begin{tikzpicture}
	
	\node (A) {$\bar{Z}_n^\gamma$};
	\node (B) [node distance=2cm, right of=A] {};
	\node (C) [node distance=2cm, right of=B] {$Hom(H^n(C^*_\gamma);G)$};
	\node (B') [node distance=2cm, below of=B] {$\bar{H}_n(C^*_\gamma ; G)$.};

	\draw[->] (A) to node [above]{$\xi_\gamma$}(C);
	\draw[->] (A) to node [above]{$\bar{p}_\gamma$}(B');
	\draw[->] (B') to node [above]{$\bar{\xi}_\gamma$}(C);
	\end{tikzpicture}
	\end{equation}
Therefore, if we take ${\varprojlim} ^ {(i)}$ of this diagram, then by Corollary 2, we obtain the following result. 

\begin{corollary}
	For each integer $i\ge 1$, ${\varprojlim} ^ {(i)} \bar{Z}_n^\gamma $ is a direct summand of ${\varprojlim} ^ {(i)} \bar{H}_n(C^*_\gamma ;G) $ and the projection of ${\varprojlim} ^ {(i)} \bar{H}_n(C^*_\gamma ;G) $ onto ${\varprojlim} ^ {(i)}\bar{Z}_n^\gamma  $ is natural.
\end{corollary}

Finally, we obtain the following important property {\color{black} of the} $\varprojlim ^{(i)}$ functor.

\begin{theorem}
	For each integer $i\ge 0$, there is a short exact sequence
	\begin{equation}\label{eq54}
	\begin{tikzpicture}
	
	\node (A) {$0$};
	\node (B) [node distance=2.5cm, right of=A] {${\varprojlim} ^ {(i)} Ext(H^{n+1}(C^*_\gamma);G) $};
	\node (C) [node distance=4cm, right of=B] {${\varprojlim} ^ {(i)} \bar{H}_n(C^*_\gamma;G)  $};
	\node (D) [node distance=4cm, right of=C] {${\varprojlim} ^ {(i)} Hom(H^n(C^*_\gamma);G) $};
	\node (E) [node distance=2.5cm, right of=D] {$0,$};
	
	\draw[->] (A) to node [above]{}(B);
	\draw[->] (B) to node [above]{}(C);
	\draw[->] (C) to node [above]{$\xi_\gamma$}(D);
	\draw[->] (D) to node [above]{}(E);
	
	\end{tikzpicture}
	\end{equation}
and this sequence splits naturally for $i\ge 1$.	
\end{theorem}
\begin{proof} Using the commutative diagram \eqref{eq53}, for each $\gamma$ we have a commutative diagram with exact rows:
	\begin{equation}\label{eq55}	
	\begin{tikzpicture}
	
	\node (A) {$0$};
	\node (B) [node distance=2cm, right of=A] {$Ker \xi _\gamma$};
	\node (C) [node distance=2.5cm, right of=B] {$\bar{Z}_n^\gamma$};
	\node (D) [node distance=2.5cm, right of=C] {$Hom(H^n(C^*_\gamma) ;G)$};
	\node (E) [node distance=2cm, right of=D] {$0$};
	
	\draw[->] (A) to node [left]{}(B);
	\draw[->] (B) to node [above]{}(C);
	\draw[->] (C) to node [above]{$\xi_\gamma$}(D);
	\draw[->] (D) to node [left]{}(E);
	
	\node (A1) [node distance=2cm, below of=A] {$0$};
	\node (B1) [node distance=2cm, below of=B] {$Ext(H^{n+1}(C^*_\gamma) ;G)$};
	\node (C1) [node distance=2cm, below of=C] {$\bar{H}_n(C^*_\gamma;G)$};
	\node (D1) [node distance=2cm, below of=D] {$Hom(H^n(C^*_\gamma) ;G)$};
	\node (E1) [node distance=2cm, below of=E] {$0$.};
	
	\draw[->] (A1) to node [left]{}(B1);
	\draw[->] (B1) to node [above]{}(C1);
	\draw[->] (C1) to node [above]{$\bar{\xi}_\gamma$}(D1);
	\draw[->] (D1) to node [left]{}(E1);
	
	\draw[->] (B) to node [left]{}(B1);
	\draw[->] (C) to node [left]{$\bar{p}_\gamma $}(C1);
	\draw[->] (D) to node [left]{$1$}(D1);
	
	\end{tikzpicture}
	\end{equation}
	This induces the following commutative diagram with exact rows
	\begin{equation}\label{eq56}	
	\begin{tikzpicture}
	
	\node (A) {$\dots$};
	\node (B) [node distance=2.4cm, right of=A] {${\varprojlim} ^ {(i)} Ker \xi _\gamma$};
	\node (C) [node distance=3.5cm, right of=B] {${\varprojlim} ^ {(i)} \bar{Z}_n^\gamma$};
	\node (D) [node distance=3.5cm, right of=C] {${\varprojlim} ^ {(i)} Hom(H^n(C^*_\gamma) ;G)$};
	\node (E) [node distance=4cm, right of=D] {${\varprojlim} ^ {(i+1)} Ker \xi _\gamma$};
	\node (F) [node distance=3cm, right of=E] {$\dots$};
	
	\draw[->] (A) to node [left]{}(B);
	\draw[->] (B) to node [above]{}(C);
	\draw[->] (C) to node [above]{${\varprojlim} ^ {(i)} \xi_\gamma$}(D);
	\draw[->] (D) to node [left]{}(E);
	\draw[->] (E) to node [left]{}(F);
	
	\node (A1) [node distance=2cm, below of=A] {$\dots$};
	\node (B1) [node distance=2cm, below of=B] {$ {\varprojlim} ^ {(i)} Ext(H^{n+1}(C^*_\gamma) ;G)$};
	\node (C1) [node distance=2cm, below of=C] {${\varprojlim} ^ {(i)} \bar{H}_n(C^*_\gamma;G)$};
	\node (D1) [node distance=2cm, below of=D] {${\varprojlim} ^ {(i)} Hom(H^n(C^*_\gamma) ;G)$};
	\node (E1) [node distance=2cm, below of=E] {${\varprojlim} ^ {(i+1)} Ext(H^{n+1}(C^*_\gamma) ;G)$};
	\node (F1) [node distance=2cm, below of=F] {$\dots$~.};
	
	\draw[->] (A1) to node [left]{}(B1);
	\draw[->] (B1) to node [above]{}(C1);
	\draw[->] (C1) to node [above]{${\varprojlim} ^ {(i)} \bar{\xi}_\gamma$}(D1);
	\draw[->] (D1) to node [left]{}(E1);
	\draw[->] (E1) to node [left]{}(F1);
	
	\draw[->] (B) to node [left]{}(B1);
	\draw[->] (C) to node [left]{${\varprojlim} ^ {(i)} \bar{p}_\gamma $}(C1);
	\draw[->] (D) to node [left]{$1$}(D1);
	\draw[->] (E) to node [left]{}(E1);
	
	\end{tikzpicture}
	\end{equation}
	By Lemma 4, ${\varprojlim} ^ {(i)} Ker \xi_\gamma = 0,$ for $i\ge 1,$  and so the beginning of the  diagram \eqref{eq56} {\color{black} is of} the following form: 
	
	\begin{equation}\label{eq57}	
	\begin{tikzpicture}
	
	\node (A) {$0$};
	\node (B) [node distance=2.5cm, right of=A] {${\varprojlim} Ker \xi _\gamma$};
	\node (C) [node distance=3.5cm, right of=B] {${\varprojlim} \bar{Z}_n^\gamma$};
	\node (D) [node distance=3.5cm, right of=C] {${\varprojlim}  Hom(H^n(C^*_\gamma) ;G)$};
	\node (E) [node distance=2.5cm, right of=D] {$0$};
	
	\draw[->] (A) to node [left]{}(B);
	\draw[->] (B) to node [above]{}(C);
	\draw[->] (C) to node [above]{${\varprojlim}  \xi_\gamma$}(D);
	\draw[->] (D) to node [left]{}(E);
	
	\node (A1) [node distance=2cm, below of=A] {$0$};
	\node (B1) [node distance=2cm, below of=B] {$ {\varprojlim}  Ext(H^{n+1}(C^*_\gamma) ;G)$};
	\node (C1) [node distance=2cm, below of=C] {${\varprojlim}  \bar{H}_n(C^*_\gamma;G)$};
	\node (D1) [node distance=2cm, below of=D] {${\varprojlim}  Hom(H^n(C^*_\gamma) ;G)$};
	\node (E1) [node distance=2cm, below of=E] {$\dots$ ~.};
	
	\draw[->] (A1) to node [left]{}(B1);
	\draw[->] (B1) to node [above]{}(C1);
	\draw[->] (C1) to node [above]{${\varprojlim}  \bar{\xi}_\gamma$}(D1);
	\draw[->] (D1) to node [left]{}(E1);
	
	\draw[->] (B) to node [left]{}(B1);
	\draw[->] (C) to node [left]{${\varprojlim} ^ {(i)} \bar{p}_\gamma $}(C1);
	\draw[->] (D) to node [left]{$1$}(D1);
	
	\end{tikzpicture}
	\end{equation}
Therefore, the following sequence is exact 
\begin{equation}\label{eq58}
\begin{tikzpicture}

\node (A) {$0$};
\node (B) [node distance=2.5cm, right of=A] {${\varprojlim} Ext(H^{n+1}(C^*_\gamma);G) $};
\node (C) [node distance=4cm, right of=B] {${\varprojlim}  \bar{H}_n(C^*_\gamma;G)  $};
\node (D) [node distance=4cm, right of=C] {${\varprojlim}  Hom(H^n(C^*_\gamma);G) $};
\node (E) [node distance=2.5cm, right of=D] {$0$};

\draw[->] (A) to node [above]{}(B);
\draw[->] (B) to node [above]{}(C);
\draw[->] (C) to node [above]{$\xi_\gamma$}(D);
\draw[->] (D) to node [above]{}(E);

\end{tikzpicture}
\end{equation}
and {\color{black} the map} ${\varprojlim} ^ {(1)}\{Ext(H^{n+1}(C^*_\gamma);G) \} \lra {\varprojlim} ^ {(1)}\{ \bar{H}_n(C^*_\gamma;G)  \}$ is a monomorphism. Therefore, we obtain the result for $i=0$. On the other hand, for $i \ge1,$ the result follows from the commutativity of the diagram \eqref{eq56} and {\color{black}Corollaries} 2 and 3. 
\end{proof}

{\color{black} Here we formulate and give the proof of the dual version of the main theorem of \cite{12}.}  

\begin{theorem}
Let ${\mathbf C }^*=\{C^*_\gamma\}$ be a direct system of cochain complexes. Then, there is a natural exact sequence
\begin{equation}\label{eq59}
\begin{tikzpicture}

\node (A) {$\dots$};
\node (B) [node distance=2cm, right of=A] {$ {\varprojlim} ^ {(3)}   \bar{H}^\gamma_{n+2} $};
\node (C) [node distance=2.5cm, right of=B] {${\varprojlim} ^ {(1)}  \bar{H}^\gamma_{n+1} $};
\node (D) [node distance=2.5cm, right of=C] {$ \bar{H}_n \left( {\varinjlim}   C^*_\gamma ;G \right) $};
\node (E) [node distance=2.5cm, right of=D] {${\varprojlim}  \bar{H}^\gamma_{n} $};
\node (F) [node distance=2cm, right of=E] {${\varprojlim} ^ {(2)}  \bar{H}^\gamma_{n} $};
\node (H) [node distance=2cm, right of=F] {$\dots$};

\draw[->] (A) to node [above]{}(B);
\draw[->] (B) to node [above]{}(C);
\draw[->] (C) to node [above]{}(D);
\draw[->] (D) to node [above]{}(E);
\draw[->] (E) to node [above]{}(F);
\draw[->] (F) to node [above]{}(H);

\end{tikzpicture}
\end{equation}
where $\bar{H}_*^\gamma=\bar{H}_* (C^*_\gamma ;G).$
\end{theorem}
\begin{proof} By Proposition 1.2 of \cite{6}, for the inverse system $\{H_\gamma^{n+1}\},$ we have an exact sequence
	
\begin{equation}\label{eq62}
\begin{tikzpicture}

\node (A) {$0$};
\node (B) [node distance=2.5cm, right of=A] {$ {\varprojlim} ^ {(1)}   Hom(H_\gamma^{n+1} ; G) $};
\node (C) [node distance=3.5cm, right of=B] {$Ext(\varinjlim H_\gamma^{n+1}; G ) $};
\node (E) [node distance=6.5cm, right of=B] {$\varprojlim Ext(H_\gamma^{n+1} ;G)$};
\node (F) [node distance=3.5cm, right of=E] {${\varprojlim} ^ {(2)}  Hom(H_\gamma^{n+1} ; G) $};
\node (H) [node distance=2.5cm, right of=F] {$0$};

\draw[->] (A) to node [above]{}(B);
\draw[->] (B) to node [above]{}(C);
\draw[->] (C) to node [above]{}(E);
\draw[->] (E) to node [above]{}(F);
\draw[->] (F) to node [above]{}(H);

\end{tikzpicture}
\end{equation}
and 
\begin{equation}\label{eq072}
{\varprojlim} ^ {(i)} Ext(H^{n+1}_\gamma ;G) \simeq {\varprojlim} ^ {(i+2)}  Hom(H^{n+1}_\gamma ; G), \text{ for } i\ge 1.
\end{equation}
Since cohomology commutes with direct limits, we have $H^*(\varinjlim C^*_\gamma;G) \simeq \varinjlim H^*(C^*_\gamma;G)$. Therefore, if $C^* \simeq \varinjlim C^*_\gamma $, then $H^{n+1}(C^*) \simeq \varinjlim H^{n+1}_\gamma,$ where $ H^{n+1}_\gamma= H^{n+1}(C^*_\gamma ;G).$  So, we obtain an exact sequence 
\begin{equation}\label{eq63}
\begin{tikzpicture}

\node (A) {$0$};
\node (B) [node distance=2.5cm, right of=A] {$ {\varprojlim} ^ {(1)}   Hom(H_\gamma^{n+1} ; G) $};
\node (C) [node distance=3.5cm, right of=B] {$Ext( H^{n+1}(C^*); G ) $};
\node (E) [node distance=6.5cm, right of=B] {$\varprojlim Ext(H_\gamma^{n+1} ;G)$};
\node (F) [node distance=3.5cm, right of=E] {${\varprojlim} ^ {(2)}  Hom(H_\gamma^{n+1} ; G) $};
\node (H) [node distance=2.5cm, right of=F] {$0$.};

\draw[->] (A) to node [above]{}(B);
\draw[->] (B) to node [above]{}(C);
\draw[->] (C) to node [above]{}(E);
\draw[->] (E) to node [above]{}(F);
\draw[->] (F) to node [above]{}(H);

\end{tikzpicture}
\end{equation}
Note that, if $i_\gamma :C^*_\gamma \lra C^*$ is a natural map, then it induces $\pi_\gamma ; \bar{H}_*(C;G) \lra \bar{H}_n(C^*_\gamma ;G)$ map. On the other hand, by Theorem 1, the following diagram is commutative:
\begin{equation}\label{eq063}	
\begin{tikzpicture}

\node (A) {$0$};
\node (B) [node distance=2cm, right of=A] {$Ext(H^{n+1}(C^*);G)$};
\node (C) [node distance=3.5cm, right of=B] {$\bar{H}_n(C^*;G)$};
\node (D) [node distance=3.5cm, right of=C] {$Hom(H^n(C^*) ;G)$};
\node (E) [node distance=2.2cm, right of=D] {$0$};

\draw[->] (A) to node [left]{}(B);
\draw[->] (B) to node [above]{}(C);
\draw[->] (C) to node [above]{}(D);
\draw[->] (D) to node [left]{}(E);

\node (A1) [node distance=2cm, below of=A] {$0$};
\node (B1) [node distance=2cm, below of=B] {$Ext(H^{n+1}(C^*_\gamma);G)$};
\node (C1) [node distance=2cm, below of=C] {$  \bar{H}_n(C^*_\gamma;G)$};
\node (D1) [node distance=2cm, below of=D] {$  Hom(H^n(C^*_\gamma) ;G)$};
\node (E1) [node distance=2cm, below of=E] {$0$.};

\draw[->] (A1) to node [left]{}(B1);
\draw[->] (B1) to node [above]{}(C1);
\draw[->] (C1) to node [above]{}(D1);
\draw[->] (D1) to node [left]{}(E1);

\draw[->] (B) to node [left]{$\tilde{\pi}_\gamma$}(B1);
\draw[->] (C) to node [left]{$\pi_\gamma$}(C1);
\draw[->] (D) to node [left]{$\bar{\pi}_\gamma$}(D1);

\end{tikzpicture}
\end{equation}
The diagram \eqref{eq063} generates the following diagram: 
 
\begin{equation}\label{eq163}	
\begin{tikzpicture}

\node (A) {$0$};
\node (B) [node distance=2cm, right of=A] {$Ext(H^{n+1}(C^*);G)$};
\node (C) [node distance=3.5cm, right of=B] {$\bar{H}_n(C^*;G)$};
\node (D) [node distance=3.5cm, right of=C] {$Hom(H^n(C^*) ;G)$};
\node (E) [node distance=2.2cm, right of=D] {$0$};

\draw[->] (A) to node [left]{}(B);
\draw[->] (B) to node [above]{}(C);
\draw[->] (C) to node [above]{}(D);
\draw[->] (D) to node [left]{}(E);

\node (A1) [node distance=2cm, below of=A] {$0$};
\node (B1) [node distance=2cm, below of=B] {$\varprojlim Ext(H^{n+1}_\gamma;G)$};
\node (C1) [node distance=2cm, below of=C] {$\varprojlim  \bar{H}_n^\gamma$};
\node (D1) [node distance=2cm, below of=D] {$\varprojlim  Hom(H^n_\gamma ;G)$};
\node (E1) [node distance=2cm, below of=E] {$0$.};

\draw[->] (A1) to node [left]{}(B1);
\draw[->] (B1) to node [above]{}(C1);
\draw[->] (C1) to node [above]{}(D1);
\draw[->] (D1) to node [left]{}(E1);

\draw[->] (B) to node [left]{$\tilde{\pi}$}(B1);
\draw[->] (C) to node [left]{$\pi$}(C1);
\draw[->] (D) to node [left]{$\simeq$}(D1);

\end{tikzpicture}
\end{equation}
 Therefore, we have $Ker \tilde{\pi} \simeq Ker \pi$ and $Coker \tilde{\pi} \simeq Coker \pi$ and so, the following diagram is commutative: 
\begin{equation}\label{eq64}	
\begin{tikzpicture}

\node (A) {$0$};
\node (B) [node distance=2cm, right of=A] {$Ext(H^{n+1}(C);G)$};
\node (C) [node distance=3.5cm, right of=B] {$\bar{H}_n(C;G)$};
\node (D) [node distance=3.5cm, right of=C] {$Hom(H^n(C);G)$};
\node (E) [node distance=2cm, right of=D] {$0$};

\draw[->] (A) to node [left]{}(B);
\draw[->] (B) to node [above]{}(C);
\draw[->] (C) to node [above]{$$}(D);
\draw[->] (D) to node [left]{}(E);

\node (B2) [node distance=1.5cm, above of=B] {$Ker \tilde{\pi}$};
\node (C2) [node distance=1.5cm, above of=C] {$Ker \pi$};
\node (B3) [node distance=1.5cm, above of=B2] {$0$};
\node (C3) [node distance=1.5cm, above of=C2] {$0$};

\draw[->] (B3) to node [left]{}(B2);
\draw[->] (C3) to node [left]{}(C2);
\draw[->] (B2) to node [left]{}(B);
\draw[->] (C2) to node [left]{}(C);
\draw[->] (B2) to node [above]{$\simeq$}(C2);

\node (A1) [node distance=2cm, below of=A] {$0$};
\node (B1) [node distance=2cm, below of=B] {$ \varprojlim Ext(H^{n+1}_\gamma ;G)$};
\node (C1) [node distance=2cm, below of=C] {$\varprojlim \bar{H}_n^\gamma$};
\node (D1) [node distance=2cm, below of=D] {$\varprojlim Hom(H^n_\gamma ;G)$};
\node (E1) [node distance=2cm, below of=E] {$0$~.};

\draw[->] (A1) to node [left]{}(B1);
\draw[->] (B1) to node [above]{}(C1);
\draw[->] (C1) to node [above]{$$}(D1);
\draw[->] (D1) to node [left]{}(E1);

\draw[->] (B) to node [left]{$\tilde{\pi}$}(B1);
\draw[->] (C) to node [left]{$\pi $}(C1);
\draw[->] (D) to node [left]{$\simeq $}(D1);

\node (B4) [node distance=1.5cm, below of=B1] {$Coker \tilde{\pi}$};
\node (C4) [node distance=1.5cm, below of=C1] {$Coker \pi$};
\node (B5) [node distance=1.5cm, below of=B4] {$0$};
\node (C5) [node distance=1.5cm, below of=C4] {$0$};

\draw[->] (B1) to node [left]{}(B4);
\draw[->] (C1) to node [left]{}(C4);
\draw[->] (B4) to node [left]{}(B5);
\draw[->] (C4) to node [left]{}(C5);
\draw[->] (B4) to node [above]{$\simeq$}(C4);
\end{tikzpicture}
\end{equation}
Using the exact sequence \eqref{eq63} and diagram \eqref{eq64}, we obtain a four-term exact sequence:

\begin{equation}\label{eq65}
\begin{tikzpicture}

\node (A) {$0$};
\node (B) [node distance=2.5cm, right of=A] {$ {\varprojlim} ^ {(1)}   Hom(H_\gamma^{n+1} ; G) $};
\node (C) [node distance=3cm, right of=B] {$\bar{H}_n(C;G)$};
\node (E) [node distance=5cm, right of=B] {$\varprojlim \bar{H}_n^\gamma$};
\node (F) [node distance=3cm, right of=E] {${\varprojlim} ^ {(2)}  Hom(H_\gamma^{n+1} ; G) $};
\node (H) [node distance=2.5cm, right of=F] {$0$.};

\draw[->] (A) to node [above]{}(B);
\draw[->] (B) to node [above]{}(C);
\draw[->] (C) to node [above]{}(E);
\draw[->] (E) to node [above]{}(F);
\draw[->] (F) to node [above]{}(H);
\end{tikzpicture}
\end{equation}
Using the exact sequence \eqref{eq65}, Theorem 3 and isomorphism \eqref{eq072}, we obtain the following diagram, which contains the long exact sequence of the theorem:
\begin{equation}\label{eq66}
\begin{tikzpicture}

\node (A) {$0$};
\node (B) [node distance=2.5cm, right of=A] {$ {\varprojlim} ^ {(1)}   Hom(H_\gamma^{n+1} ; G) $};
\node (C) [node distance=3cm, right of=B] {$\bar{H}_n(C;G)$};
\node (E) [node distance=5cm, right of=B] {$\varprojlim \bar{H}_n^\gamma$};
\node (F) [node distance=3cm, right of=E] {${\varprojlim} ^ {(2)}  Hom(H_\gamma^{n+1} ; G) $};
\node (H) [node distance=2.5cm, right of=F] {$0$.};

\draw[->] (A) to node [above]{}(B);
\draw[->] (B) to node [above]{}(C);
\draw[->] (C) to node [above]{}(E);
\draw[->] (E) to node [above]{}(F);
\draw[->] (F) to node [above]{}(H);

\node (A1) [node distance=1.5cm, below of=A] {$0$};
\node (B1) [node distance=1.5cm, below of=B] {${\varprojlim} ^ {(2)} Ext(H_\gamma^{n+2} ; G)$};
\node (E1) [node distance=1.5cm, below of=E] {${\varprojlim} ^ {(2)} \bar{H}_{n+1}^\gamma$};
\node (F1) [node distance=1.5cm, below of=F] {${\varprojlim} ^ {(2)} Hom(H_\gamma^{n+1} ; G)$};
\node (H1) [node distance=1.5cm, below of=H] {$0$};

\draw[->] (H1) to node [above]{}(F1);
\draw[->] (F1) to node [above]{}(E1);
\draw[->] (E1) to node [above]{}(B1);
\draw[->] (B1) to node [above]{}(A1);

\node (A2) [node distance=1.5cm, below of=A1] {$0$};
\node (B2) [node distance=1.5cm, below of=B1] {${\varprojlim} ^ {(4)} Hom(H_\gamma^{n+2} ; G)$};
\node (E2) [node distance=1.5cm, below of=E1] {${\varprojlim} ^ {(4)} \bar{H}_{n+2}^\gamma$};
\node (F2) [node distance=1.5cm, below of=F1] {${\varprojlim} ^ {(4)} Ext(H_\gamma^{n+3} ; G)$};
\node (H2) [node distance=1.5cm, below of=H1] {$0$};

\node (EE2) [node distance=1.5cm, below of=E2] {$\vdots$};

\draw[->] (A2) to node [above]{}(B2);
\draw[->] (B2) to node [above]{}(E2);
\draw[->] (E2) to node [above]{}(F2);
\draw[->] (F2) to node [above]{}(H2);

\node (A3) [node distance=1.5cm, above of=A] {$0$};
\node (B3) [node distance=1.5cm, above of=B] {${\varprojlim} ^ {(1)} Hom(H_\gamma^{n+1} ; G)$};
\node (E3) [node distance=1.5cm, above of=E] {${\varprojlim} ^ {(1)} \bar{H}_{n+1}^\gamma$};
\node (F3) [node distance=1.5cm, above of=F] {${\varprojlim} ^ {(1)} Ext(H_\gamma^{n+2} ; G)$};
\node (H3) [node distance=1.5cm, above of=H] {$0$};

\draw[->] (H3) to node [above]{}(F3);
\draw[->] (F3) to node [above]{}(E3);
\draw[->] (E3) to node [above]{}(B3);
\draw[->] (B3) to node [above]{}(A3);

\node (A4) [node distance=1.5cm, above of=A3] {$0$};
\node (B4) [node distance=1.5cm, above of=B3] {${\varprojlim} ^ {(3)} Ext(H_\gamma^{n+3} ; G)$};
\node (E4) [node distance=1.5cm, above of=E3] {${\varprojlim} ^ {(3)} \bar{H}_{n+2}^\gamma$};
\node (F4) [node distance=1.5cm, above of=F3] {${\varprojlim} ^ {(3)} Hom(H_\gamma^{n+2} ; G)$};
\node (H4) [node distance=1.5cm, above of=H3] {$0$};

\node (EE4) [node distance=1.5cm, above of=E4] {$\vdots$};

\draw[->] (A4) to node [above]{}(B4);
\draw[->] (B4) to node [above]{}(E4);
\draw[->] (E4) to node [above]{}(F4);
\draw[->] (F4) to node [above]{}(H4);

\draw[dotted, thick, ->] (EE4) to node [above]{}(E4);
\draw[dotted, thick, ->] (E4) to node [above]{}(E3);
\draw[dotted, thick, ->] (E3) to node [above]{}(C);
\draw[dotted, thick, ->] (E) to node [above]{}(E1);
\draw[dotted, thick, ->] (E1) to node [above]{}(E2);
\draw[dotted, thick, ->] (E2) to node [above]{}(EE2);

\draw[->] (F4) to node [left]{$\simeq $}(F3);
\draw[->] (B3) to node [left]{$\simeq $}(B);
\draw[->] (F) to node [left]{$\simeq $}(F1);
\draw[->] (B1) to node [left]{$\simeq $}(B2);
\end{tikzpicture}
\end{equation}
\end{proof}

\begin{corollary}
	Let ${\mathbf C }^*=\{C^*_\gamma\}$ be a direct system of cochain complexes. Then, for each injective group $G,$ there is an isomorphism
	\begin{equation}\label{eq59}
	\bar{H}_n \left( {\varinjlim}   C^*_\gamma ;G \right)  \simeq {\varprojlim} \bar{H}_* (C^*_\gamma ;G).
	\end{equation}
\end{corollary}

\section{Applications {\color{black}in homology theory}} 
{\bf 1. }Let  $C_c^*(X,G)$ be the cochain complex of Massey \cite{Mas}. It is known that for each locally compact Hausdorff space $X$ and each integer $n$ the cochain group $C_c^n(X,\bZ)$ with integer coefficient is a free abelian group ({\color{black}Theorem 4.1} \cite{Mas}). Using the cochain complex  $C_c^*(X,G)$, Massey defined an exact homology $H_*^M$, the so called Massey homology on the category of locally compact spaces and proper maps as a homology of the chain complex $C_*(X,G)=\Hom (C_c^*(X),G)$. Consequently, for the given category, the Universal Coefficient Formula is obtained (see {\color{black} Theorem 4.1, Corollary 4.18} \cite{Mas} and {\color{black} Theorem 4.1} \cite{Mac}):
\begin{equation}\label{eq93}
0 \lra \Ext(H_c^{n+1}(X),G) \lra {H}^M_n(X,G) \lra \Hom(H_c^n(X),G) \lra 0.
\end{equation}
Let  $C^*_c=C^*_c(X;G)$ be the cochain complex of Massey. Consider the chain complex $\bar{C}^M_*(X;G)=Hom(C^*(X);\beta_{\#})$. Let $\bar{H}^M_*(X;G)$ be homology of the chain complex $\bar{C}^M_*(X;G)$. In this case, by Theorem 1 we 
will obtain the Universal Coefficient Formula
\begin{equation}\label{eq94}
0 \lra \Ext(H_c^{n+1}(X),G) \lra \bar{H}^M_n(X,G) \lra \Hom(H_c^n(X),G) \lra 0.
\end{equation}
Note that by Theorem 2, for the category of locally compact spaces the homologies $\bar{H}^M_n(X,G)$ and $H^M_n(X,G)$ are isomorphic. 

Note that for the Massey homology theory $\bar{H}^M_*(-;G)$ our construction gives the following result:

\begin{corollary}
	Let $X$ be a locally compact Hausdorff space, then\\
	a) if  $\{N_\alpha\}$ is the system of closed neighborhoods $N_\alpha$ of closed subspace $A$ of $X$, directed by inclusion, then it induces the following exact sequence:
	$$\cdots \lra \llm^{(2k+1)} \bar{H}^M_{n+k+1}(N_\alpha) \lra \cdots \llm^{(3)} \bar{H}^M_{n+2}(N_\alpha) \lra \llm^{(1)} \bar{H}^M_{n+1}(N_\alpha) \lra$$
	\begin{gather}
	\lra \bar{H}^M_{n}(A,G) \os{i_*}{\lra} \llm \bar{H}^M_{n}(N_\alpha) \lra  \llm^{(2)} \bar{H}^M_{n+1}(N_\alpha) \lra \cdots \lra \llm^{(2k)} \bar{H}^M_{n+k}(N_\alpha) \lra \cdots\, .
	\end{gather}
	b) if  $\{U_\alpha\}$ is the system of open subspaces of $X$, such that $\bar{U}_\alpha$ is compact and $X=\bigcup U_\alpha$ directed by inclusion, then it induces  the following exact sequence:
	$$\cdots \lra \llm^{(2k+1)} \bar{H}^M_{n+k+1}(U_\alpha) \lra \cdots \llm^{(3)} \bar{H}^M_{n+2}(U_\alpha) \lra \llm^{(1)} \bar{H}^M_{n+1}(U_\alpha) \lra$$
	\begin{gather}\label{eq89.1}
	\lra \bar{H}^M_{n}(X,G) \os{i_*}{\lra} \llm \bar{H}^M_{n}(U_\alpha) \lra  \llm^{(2)} \bar{H}^M_{n+1}(U_\alpha) \lra \cdots \lra \llm^{(2k)} \bar{H}^M_{n+k}(U_\alpha) \lra \cdots\, .
	\end{gather}
\end{corollary} 

Note that the formula \eqref{eq89.1} is a generalization of Theorem 4.22 of \cite{Mas}.

{\bf 2.} Let $G$ be an $R$-module over a principal ideal domain $R$ and let $X$ be a topological space. Denote by $\bar{C}^*(X;G) $ the cochain complex of Alexander-Spanier \cite{15} and by $\bar{H}^*(X;G)$ the Alexander-Spanier cohomology. let $A$ be a subspace of a topological space $X$ and $\{U_\alpha \}$ be the family of all neighborhoods of $A$ in $X$ directed downward by inclusion. Hence, $\{\bar{H}^n(U_\alpha;G) \}$ is a direct system. The restriction maps $\bar{H}^n(U_\alpha;G) \lra \bar{H}^n(A;G)$ define a natural homomorphism
\begin{equation}\label{eq90}
i: \varinjlim \bar{H}^n (U_\alpha;G) \lra \bar{H}^n(A;G).
\end{equation}
By Theorem 6.6.2 \cite{15}, if $A$ is a closed subspace of a paracompact Hausdorff space $X$, then \eqref{eq90} is an isomorphism. In this case, $A$ is called a taut subspace relative to the Alexander-Spanier cohomology theory. In the case of homology theory, we have a natural homomorphism 
\begin{equation}\label{eq91}
i:  {H}_n(A;G) \lra \varprojlim {H}_n (U_\alpha;G).
\end{equation}
The question whether the homomorphism \eqref{eq91} is an isomorphism or not  was open.

Let $\bar{C}^*=\bar{C}^*(X;G)$ be the cochain complex of Alexander-Spanier. Consider the chain complex $\bar{C}_*(X;G)=Hom(\bar{C}^*(X);\beta_{\#})$. Let $\bar{H}_*(X;G)$ be the homology of the chain complex $\bar{C}_*(X;G)$. In this case, we will say that the homology $\bar{H}_*(X;G)$ is generated by the Alexander-Spanier cochains $\bar{C}^*(X;G)$. By Theorem 4 we have the long exact sequence, which contains the homomorphisms \eqref{eq91}.

\begin{corollary}
	If $A$ is a closed subspace of a paracompact Hausdorff space $X$ and $\{U_\alpha\}$ is the family of all neighborhoods of $A$ in $X$, then there is a long exact sequence:
	$$\cdots \lra \llm^{(2k+1)} \bar{H}_{n+k+1}(U_\alpha) \lra \cdots \llm^{(3)} \bar{H}_{n+2}(U_\alpha) \lra \llm^{(1)} \bar{H}_{n+1}(U_\alpha) \lra$$
	\begin{gather}
	\lra \bar{H}_{n}(A;G) \os{i_*}{\lra} \llm \bar{H}_{n}(U_\alpha) \lra  \llm^{(2)} \bar{H}_{n+1}(U_\alpha) \lra \cdots \lra \llm^{(2k)} \bar{H}_{n+k}(U_\alpha) \lra \cdots\, .
	\end{gather}
\end{corollary}

{\bf 3.} It is clear that there is a natural inclusion $i^\#:C^*_c(X;G) \to  \bar{C}_*(X;G)$ from the Massey cochain complex to the Alexander-Spanier cochain complex, which induces the corresponding homomorphism $i^*:\bar{H}_*(-;G) \to \bar{H}^M_*(-;G),$  where $\bar{H}_*(-;G)$ and $\bar{H}^M_*(-;G)$  are homologies generated by the Alexander-Spanier and the Massey cochains, respectively. Therefore,  $\bar{H}_*(-;G)$ and  $\bar{H}^M_*(-;G)$  are homologies of the chain complexes $\bar{C}_*(-;G)=Hom(\bar{C}^*(-);\beta_\#)$ and $\bar{C}_*^M(-;G)=Hom(C^*_c(-);\beta_\#).$  On the other hand, on the category of compact Hausdorff spaces, the Alexsander-Spanier and the Massey cohomology are isomorphic and by the Universal Coefficient Formula, we will obtain that for each compact Hausdorff {\color{black} space there} is an isomorphism:
\begin{equation}\label{eq31.0}
i^*:\bar{H}_*(X;G) \os{\simeq}{\lra} \bar{H}^M_*(X;G).
\end{equation}
On the other hand, since on the category of compact metric spaces the Steenrod homology $H^{St}_*$ and the Massey homology are isomorphic, using the isomorphism \eqref{eq31.0}, we will obtain that
\begin{equation}\label{eq31.1}
\bar{H}_*(X;G) \simeq {H}^{St}_*(X;G).
\end{equation}
The same way, on the category of compact Hausdorff spaces, the Milnor homology $H^{Mil}_*$ and the Massey homology are isomorphic and consequently, we have
\begin{equation}\label{eq31.2}
\bar{H}_*(X;G) \simeq {H}^{Mil}_*(X;G).
\end{equation}
If $H_*^{BM}(-;G)$ is the Borel-Moore homology with coefficients in $G$, then by Theorem 3 \cite{Kuz}, we have the {\color{black}isomorphism}
\begin{equation}\label{}
\bar{H}_*(X;G) \simeq {H}^{BM}_*(X;G).
\end{equation}

{\bf 4.} Let $K_C$ be the category of compact pairs $(X,A)$ and continuous maps and $H_*$ be an exact homology theory. Let $\{(X_\alpha , A_\alpha) \}$ be an inverse system of compact pairs $(X_\alpha, A_\alpha)$ and $(X,A)= \varprojlim (X_\alpha, A_\alpha).$ The inverse system $\{(X_\alpha , A_\alpha) \}$ generates an inverse system $\{H_*(X_\alpha , A_\alpha) \}$ and the projections $\pi_\alpha :(X,A) \to (X_\alpha , A_ \alpha)$ induce the homomorphisms $\pi_{\alpha, *} :H_*(X,A) \to H_*(X_\alpha , A_ \alpha),$ which induce the homomorphism
\begin{equation}\label{eq92}
\pi_{*} :H_*(X,A) \to \varprojlim H_*(X_\alpha , A_ \alpha).
\end{equation}
\begin{definition} 
An exact homology theory $H_*$ {\color{black} is said to be continuous} on the category $K_C$, if for each inverse system $\{(X_\alpha , A_\alpha) \}$ of the given category, there is an infinite exact sequence
$$\cdots \lra \llm^{(2k+1)} {H}_{n+k+1}(X_\alpha , A_\alpha) \lra \cdots \llm^{(3)} {H}_{n+2}(X_\alpha , A_\alpha) \lra \llm^{(1)} {H}_{n+1}(X_\alpha , A_\alpha) \lra$$
\begin{gather}
\lra {H}_{n}(X,A;G) \os{\pi_*}{\lra} \llm {H}_{n}(X_\alpha , A_\alpha) \lra  \llm^{(2)} {H}_{n+1}(X_\alpha , A_\alpha) \lra \cdots \lra \llm^{(2k)} {H}_{n+k}(X_\alpha , A_\alpha) \lra \cdots\, .
\end{gather}
\end{definition}

\begin{definition}
	A direct system $\mathbf{C}^*=\{C^*_\alpha \}$ of the cochain complexes $C^*_\alpha$ is said to be associated with a cochain complex $C^*,$ if there is a homomorphism $\mathbf{C}^* \to C^*$ such that for each $n \in \mathbb{Z}$ the induced homomorphism
	\begin{equation}\label{eq93}
	 \varinjlim H^*(C^*_\alpha) \to H^*(C^*)
	\end{equation}
	is an isomorphism.
\end{definition}

\begin{lemma}
	If a direct system $\mathbf{C}^*=\{C^*_\alpha \}$ of the cochain complexes $C^*_\alpha$ is associated with a cochain complex $C^*$, then there is an infinite exact sequence 
	$$\cdots \lra \llm^{(2k+1)} \bar{H}_{n+k+1}(C^*_\alpha) \lra \cdots \llm^{(3)} \bar{H}_{n+2}(C^*_\alpha) \lra \llm^{(1)} \bar{H}_{n+1}(C^*_\alpha) \lra$$
	\begin{gather} \label{lem1}
	\lra \bar{H}_{n}(C^*;G) \os{\pi_*}{\lra} \llm \bar{H}_{n}(C^*_\alpha) \lra  \llm^{(2)} \bar{H}_{n+1}(C^*_\alpha) \lra \cdots \lra \llm^{(2k)} \bar{H}_{n+k}(C^*_\alpha) \lra \cdots\, .
	\end{gather}
	where $\bar{H}_{*}(C^*)=H_*(Hom(C^*;\beta_\#)$ and $\bar{H}_{*}(C^*_\alpha)=H_*(Hom(C^*_\alpha;\beta_\#),$
\end{lemma}

\begin{proof}
By theorem 4, {\color{black}there} is a natural exact sequence
	\begin{equation}\label{lem2}
	\begin{tikzpicture}
	
	\node (A) {$\dots$};
	\node (B) [node distance=2cm, right of=A] {$ {\varprojlim} ^ {(3)}   \bar{H}^\alpha_{n+2} $};
	\node (C) [node distance=2.5cm, right of=B] {${\varprojlim} ^ {(1)}  \bar{H}^\alpha_{n+1} $};
	\node (D) [node distance=2.5cm, right of=C] {$ \bar{H}_n \left( {\varinjlim}   C^*_\alpha ;G \right) $};
	\node (E) [node distance=2.5cm, right of=D] {${\varprojlim}  \bar{H}^\alpha_{n} $};
	\node (F) [node distance=2cm, right of=E] {${\varprojlim} ^ {(2)}  \bar{H}^\alpha_{n} $};
	\node (H) [node distance=2cm, right of=F] {$\dots$~,};
	
	\draw[->] (A) to node [above]{}(B);
	\draw[->] (B) to node [above]{}(C);
	\draw[->] (C) to node [above]{}(D);
	\draw[->] (D) to node [above]{}(E);
	\draw[->] (E) to node [above]{}(F);
	\draw[->] (F) to node [above]{}(H);
	
	\end{tikzpicture}
	\end{equation}
	where $\bar{H}_*^\alpha=\bar{H}_* (C^*_\alpha ;G).$ Since the direct system ${\bf C}^*=\{C^*_\alpha\}$ of cochain complexes $C^*_\alpha $ is associated with a cochain complex $C^*$, there is an isomorphism 
	\begin{equation}\label{lem3}
	 \bar{H}_*(\varinjlim C^*_\alpha;G) \simeq  \varinjlim \bar{H}^*(C^*_\alpha;G) \os{\simeq}{\lra} \bar{H}_*(C^*;G).
	 \end{equation}
	 On the other hand, by Universal Coefficient Formula, we have the following commutative diagram with exact rows:
	 \begin{equation}\label{lem4}	
	 \begin{tikzpicture}
	 
	 \node (A) {$0$};
	 \node (B) [node distance=2cm, right of=A] {$Ext(H^{n+1}(C^*);G)$};
	 \node (C) [node distance=3.5cm, right of=B] {$\bar{H}_n(C^*;G)$};
	 \node (D) [node distance=3.5cm, right of=C] {$Hom(H^n(C^*) ;G)$};
	 \node (E) [node distance=2.2cm, right of=D] {$0$};
	 
	 \draw[->] (A) to node [left]{}(B);
	 \draw[->] (B) to node [above]{}(C);
	 \draw[->] (C) to node [above]{}(D);
	 \draw[->] (D) to node [left]{}(E);
	 
	 \node (A1) [node distance=2cm, below of=A] {$0$};
	 \node (B1) [node distance=2cm, below of=B] {$Ext(H^{n+1}(\varinjlim C^*_\alpha);G)$};
	 \node (C1) [node distance=2cm, below of=C] {$  \bar{H}_n(\varinjlim C^*_\alpha;G)$};
	 \node (D1) [node distance=2cm, below of=D] {$  Hom(H^n(\varinjlim C^*_\alpha) ;G)$};
	 \node (E1) [node distance=2cm, below of=E] {$0$.};
	 
	 \draw[->] (A1) to node [left]{}(B1);
	 \draw[->] (B1) to node [above]{}(C1);
	 \draw[->] (C1) to node [above]{}(D1);
	 \draw[->] (D1) to node [left]{}(E1);
	 
	 \draw[->] (B) to node [left]{$\simeq$}(B1);
	 \draw[->] (C) to node [left]{$\bar{\pi}_n$}(C1);
	 \draw[->] (D) to node [left]{$\simeq$}(D1);
	 
	 \end{tikzpicture}
	 \end{equation}
	 Hence, the homomorphism $\bar{\pi}_n$ is an isomorphism for all $n \in \mathbb{Z}.$ Using the exact sequence \eqref{lem2} and the isomorphism $\bar{\pi}_n,$ we obtain an infinite exact sequence \eqref{lem1}.	 
\end{proof}

\begin{corollary}
    Let $\{(X_\alpha, A_\alpha )\}$ be an inverse system of pairs of compact spaces $(X_\alpha, A_\alpha ) $ and $ (X, A)= \varprojlim (X_\alpha, A_\alpha ) $. If $\bar{H}_*$ is the homology theory generated by the Alexander-Spanier cochains, then there is an infinite exact sequence
    $$\cdots \lra \llm^{(2k+1)} \bar{H}_{n+k+1}(X_\alpha, A_\alpha ) \lra \cdots \llm^{(3)} \bar{H}_{n+2}(X_\alpha, A_\alpha ) \lra \llm^{(1)} \bar{H}_{n+1}(X_\alpha, A_\alpha ) \lra$$
    \begin{gather}
    \lra \bar{H}_{n}(X,A;G) \os{\pi_*}{\lra} \llm \bar{H}_{n}(X_\alpha, A_\alpha ) \lra  \llm^{(2)} \bar{H}_{n+1}(X_\alpha, A_\alpha ) \lra \cdots \lra \llm^{(2k)} \bar{H}_{n+k}(X_\alpha, A_\alpha ) \lra \cdots\, .
    \end{gather}
\end{corollary}

\begin{corollary}
	Let $\{(X_i,A_i)\}_{i \in \mathbb{Z}}$ be an inverse {\color{black} sequence} of compact metric spaces $(X_i,A_i) $ and $ (X,A)= \varprojlim (X_i,A_i) $. If $\bar{H}_*$ is the homology theory generated by the Alexander-Spanier cochains, then there is an exact sequence
	\begin{gather}
	0 \lra \llm^{(1)} \bar{H}_{n+1}(X_i,A_i) \lra \bar{H}_{n}(X,A;G) \os{\pi_*}{\lra} \llm \bar{H}_{n}(X_i,A_i) \lra  0 .
	\end{gather}
\end{corollary}

\begin{corollary}
	 If $\bar{H}_*$ is the homology theory generated by the Alexander-Spanier cochains, then there is an exact sequence
	\begin{gather}
	0 \lra \llm^{(1)} \bar{H}_{n+1}(K_\alpha,L_\alpha) \lra \bar{H}_{n}(X,A;G) \os{\pi_*}{\lra} \llm \bar{H}_{n}(K_\alpha, L_\alpha) \lra  0 ,
	\end{gather}
	where $ (X,A)=\varprojlim (K_\alpha,L_\alpha)$ and $(K_\alpha,L_\alpha)$ are finite polyhedral pairs.
\end{corollary}

{\bf 5.} Let $C^*_s(X;G)$ be the singular cochain complex of topological spaces $X$ and  $\bar{C}^s_*(X;G)=Hom(C^*_s(X);\beta_\#).$ Let $\bar{H}_*^s(X;G)$  be the homology of the obtained chain complex $\bar{C}^s_*(X;G).$ Therefore, $\bar{H}_*^s(-;G)$ is the homology generated by the singular cochain complex  ${C}^*_s(-;G).$ It is known that there is a natural homomorphism $J^\#:\bar{C}^*(X:G) \to C^*_s(X;G)$ from the Alexander-Sapnier cochain complex to the singular cochain complex, which induces the isomorphism $j^*:\bar{H}^*(X;G) \to \bar{H}^*_s(X;G)$ on the category of manifolds. Therefore, by the Universal Coefficient Formula, we will obtain that if $X$ is manifold, then there is an isomorphism:
\begin{equation}\label{eq31}
j_*:\bar{H}^s_*(X;G) \os{\simeq}{\lra} \bar{H}_*(X;G).
\end{equation}

{\color{black}\section*{Acknowledgement.}
	
	The work partially was supported by Shota Rustaveli National Science Foundation of Georgia (SRNSF grant FR-23-271).\\
	
	We extend our sincere thanks to the reviewers for their valuable comments and suggestions, which have enhanced the exploration of our results.}
\\
~~
\\
~~
{\bf Compliance with Ethical conduct:} Not applicable.\\
{\bf Conflict of interest:} The authors declare no competing interests.\\
{\bf Data Availability:} Not applicable.

\bibliographystyle{elsarticle-num}

\end{document}